\documentclass[11pt]{amsart}
\usepackage{amssymb,amsmath,amsthm}
\usepackage{verbatim}
\usepackage{enumerate}
\usepackage{graphicx}
\usepackage{epsfig}
\usepackage{color, soul}
\usepackage[all]{xy}
\usepackage{wasysym}
\usepackage{url}
\usepackage{mathrsfs}
\usepackage{bbm}

\DeclareFontFamily{U}{mathx}{}
\DeclareFontShape{U}{mathx}{m}{n}{<-> mathx10}{}
\DeclareSymbolFont{mathx}{U}{mathx}{m}{n}
\DeclareMathAccent{\widehat}{0}{mathx}{"70}
\DeclareMathAccent{\widecheck}{0}{mathx}{"71}

\DeclareMathOperator{\spa}{span}
\DeclareMathOperator{\cb}{cb}
\DeclareMathOperator{\jcb}{jcb}
\DeclareMathOperator{\CB}{\mathcal{CB}}

\DeclareMathOperator{\amconv}{amconv}

\DeclareMathOperator{\MIN}{MIN}
\DeclareMathOperator{\MAX}{MAX}

\newcommand{\Hcb}{\mathcal{H}^\infty_{\cb}}
\newcommand{\Gcb}{\mathcal{G}^\infty_{\cb}}

\newcommand{\n}[1]{ \left\|#1\right\| }
\newcommand{\bn}[1]{ \big\|#1\big\| }

\newcommand{\tn}[1]{{\left\vert\kern-0.25ex\left\vert\kern-0.25ex\left\vert #1 
    \right\vert\kern-0.25ex\right\vert\kern-0.25ex\right\vert}}
\newcommand{\N}{{\mathbb{N}}}

\newcommand{\C}{{\mathbb{C}}}
\newcommand{\D}{{\mathbb{D}}}
\newcommand{\T}{{\mathbb{T}}}
\newcommand{\K}{{\mathcal{K}}}

\newcommand{\pair}[2]{{\langle #1, #2 \rangle}}
\newcommand{\mpair}[2]{{\langle\langle #1, #2 \rangle\rangle}}

\newtheorem{theorem}{Theorem}[section]
\newtheorem{lemma}[theorem]{Lemma}
\newtheorem{definition}[theorem]{Definition}
\newtheorem{corollary}[theorem]{Corollary}
\newtheorem{proposition}[theorem]{Proposition}

\theoremstyle{definition}
\newtheorem{remark}[theorem]{Remark}
\newtheorem{example}[theorem]{Example}
\newtheorem{question}[theorem]{Question}

\numberwithin{equation}{section}

\def\ker{{\rm ker\, }}

\keywords{Holomorphic functions, Operator spaces, Linearization}

\subjclass{46G20, 46T25, 47L25, 46L07}

\begin{document}

\title[Linearizing holomorphic functions on operator spaces]{Linearizing holomorphic functions on operator spaces}

\author[J.A. Ch\'avez-Dom\'inguez]{Javier Alejandro Ch\'avez-Dom\'inguez}
\address{Department of Mathematics, University of Oklahoma, Norman, OK 73019-3103,
USA} \email{jachavezd@ou.edu}

\author[V. Dimant]{Ver\'onica Dimant}
\address{Departamento de Matem\'{a}tica y Ciencias, Universidad de San
Andr\'{e}s, Vito Dumas 284, (B1644BID) Victoria, Buenos Aires,
Argentina and CONICET} \email{vero@udesa.edu.ar}

\thanks{The first-named author was partially supported by NSF grants DMS-1900985 and DMS-2247374. 
The second-named author was partially supported by CONICET PIP 11220200101609CO and ANPCyT PICT 2018-04104.
}

\date{}

\begin{abstract}
We introduce a notion of completely bounded holomorphic functions defined on the open unit ball of an operator space.
We endow the set of these functions with an operator space structure, and in the scalar-valued case we identify an operator space predual for it which is a noncommutative version of Mujica's predual for the space of bounded holomorphic functions and satisfies similar properties.
In particular, our predual is a free holomorphic operator space in the sense that it satisfies a linearization property for vector-valued completely bounded holomorphic functions.
Additionally, several different operator space approximation properties transfer between the predual and the domain.
\end{abstract}

\maketitle

\section{Introduction}

Operator spaces are a quantized or noncommutative version of Banach spaces, with origins and important applications in the theory of operator algebras. Recall that an operator space $V$ is a Banach space with an additional structure, namely a sequence of norms on the spaces $M_n(V)$ of $n \times n$ matrices with entries in $V$ (satisfying certain consistency requirements). A morphism between operator spaces $V$ and $W$ is no longer simply a bounded linear mapping $T: V \to W$, but one that is \emph{completely bounded}, meaning that all their amplifications $T_n : M_n(V) \to M_n(W)$ defined by applying $T$ entrywise are uniformly bounded. 
The systematic study of operator spaces started with Ruan's thesis, and was developed mainly by Effros and Ruan, Blecher and Paulsen, Pisier and Junge (see the monographs \cite{Junge-Habilitationschrift,Pisier-Asterisque,Effros-Ruan-book,Pisier-Operator-Space-Theory,BlecherLeMerdy2004} and the references therein).
Due to the nature of the definition of an operator space, for any aspect of Banach space theory it makes sense to investigate a noncommutative counterpart in the operator space world.
The linear and operator algebraic aspects have already been thoroughly developed, as one can see by taking a look at the aforementioned references, and the present paper aims to take a step in the nonlinear direction.
An interesting example of nonlinear maps which has been extensively studied in Banach spaces is the polynomials, see e.g. the surveys \cite{Gutierrez-Jaramillo-Llavona,Zalduendo}. Polynomials in operator spaces have already been considered, although the literature in this direction is limited: the main references are \cite{Dineen-Radu-polynomials,Defant-Wiesner}. We point out that \cite{Defant-Wiesner} illustrates very well a common theme when developing operator space versions of Banach space notions: in the noncommutative world there may be several different natural notions corresponding to a single classical one.
Once one has considered polynomials, a next obvious step is to look at holomorphic mappings.
The main goal of the present paper is precisely to initiate the study of  bounded holomorphic mappings in the context of operator spaces.

There is some precedent for considering holomorphic mappings on operator spaces. In \cite{Wang-Wong}, their approach is to consider the Taylor series at $0$ and require that the polynomials appearing in it are completely bounded in the sense of \cite{Dineen-Radu-polynomials}. While this notion was enough for the purposes of that paper, it does not appear well-suited for the study of \emph{spaces} of holomorphic mappings on operator spaces because there is no clear way to define even a norm based on that definition (let alone an operator space structure). This is of course a common issue with holomorphic mappings, to quote \cite{Carando-Lassalle-Zalduendo} ``spaces of holomorphic functions are seldom Banach''.
In this paper we propose an approach that successfully produces not only a Banach space, but in fact an operator space: a noncommutative version of the classical space $\mathcal{H}^\infty(B_X,Y)$ of bounded $Y$-valued holomorphic mappings defined on the unit ball of a Banach space $X$.

Our approach is to mimic the definition of a completely bounded linear map and look at amplifications: for operator spaces $V$ and $W$, we say that a holomorphic function $f : B_V \to W$ is \emph{completely bounded}
\footnote{
We warn the reader that the same terminology was used in \cite{Wang-Wong} for a different notion. Note, however, that the holomorphic functions considered in \cite{Wang-Wong} are not even necessarily bounded.
}
if its amplifications are uniformly bounded on the unit balls $B_{M_n(V)}$ of the spaces of $V$-valued matrices (note that this in fact reduces to the notion of a completely bounded linear mapping when $f$ is the restriction of a linear mapping $V \to W$). We denote this space by $\Hcb(B_V,W)$ and show that it is not only  a Banach space  but  in fact has a natural operator space structure (Lemma \ref{lemma-Hcb-Banach}), and moreover it is a completely contractive Banach algebra (Proposition \ref{proposition-Hinfty-is-algebra}).
Completely bounded holomorphic functions turn out to be well-suited for operator space theory: just as complex Banach spaces with biholomorphic unit balls must be isomorphic \cite{Kaup-Upmeier}, if the biholomorphism is completely bounded in both directions then the operator spaces have to be completely isometrically isomorphic (Proposition 3.14).  
Moreover, our notion of completely bounded holomorphic functions encompasses previously considered classes of polynomials on operator spaces from \cite{Dineen-Radu-polynomials,Defant-Wiesner}.

We remark that while our approach here is inspired by the completely bounded linear maps, the idea of applying a function (and particularly a holomorphic one) to the elements of a matrix is of course not new. See \cite[Sec. 6.3]{Horn-Johnson} for more on this subject in the case of (finite) matrices with complex entries, where it is pointed out that this idea arises naturally in the study of discretizations of integral operators, integral equations, and integral quadratic forms.
There is also an extensive literature on maps (even between $C^*$-algebras) whose amplifications preserve positivity or are even monotone, see \cite{MR4378092,MR4440677} for two recent examples.
Amplifications of nonlinear functions on operator spaces, from a more metric point of view, have recently been considered in \cite{Braga-CD,Braga-CD-Sinclair,Braga2021OpSp,Braga-Oikhberg,braga2024small}.

Our main results are that the space $\Hcb(B_V,\C)$ has an operator space predual $\Gcb(B_V)$ which satisfies operator space analogues of two important properties of its classical counterpart $\mathcal{G}^\infty(B_X)$: a universal linearization property and the transference of approximation properties.

We present two alternative ways of constructing the space $\Gcb(B_V)$. In Section \ref{Sect:Dixmier-Ng-Kaijser} we use the abstract machinery of the Dixmier-Ng-Kaijser theorem that characterizes when a Banach space is a dual space, in conjunction with a theorem of Le Merdy that identifies when an operator space which is a dual Banach space is in fact a dual operator space.
Later, we give a more explicit description of the operator space structure of $\Gcb(B_V)$ by identifying its closed matrix unit ball as the closure of the absolutely matrix convex hull of a natural set of evaluations (Proposition \ref{proposition-ball-is-convex-hull-of-image-of-delta}).

The linearization property of $\mathcal{G}^\infty(B_X)$ is the following \cite[Thm. 2.1]{Mujica-linearization} (see also the survey \cite{survey-free-objects}): there exists a bounded holomorphic mapping $\delta_X : B_X \to \mathcal{G}^\infty(B_X)$ such that for any Banach space $Y$ and any bounded holomorphic function $f : B_X \to Y$ there exists a unique bounded linear map $\widehat{f} : \mathcal{G}^\infty(B_X) \to Y$ making the following diagram commutative:
\[
\xymatrix{\mathcal{G}^\infty_{}(B_X)\ar@{-->}[rd]^{\widehat{f}}\\
     B_X\ar[r]^{f}\ar^{\delta_X}[u]& Y }
\]
We prove an analogous result for operator spaces (Theorem \ref{thm-linearization}), where the linear map $\widehat{f}$ is now completely bounded.
Let us point out that one can trivially make $\widehat{f}$ completely bounded, but in an unsatisfactory way.
If we have two operator spaces $V$ and $W$ and consider them as Banach spaces, for any bounded holomorphic mapping $f : B_V \to W$ we have the classical diagram
\[
\xymatrix{\mathcal{G}^\infty_{}(B_V)\ar@{-->}[rd]^{\widehat{f}}\\
     B_V\ar[r]^{f}\ar^{\delta_V}[u]& W }
\]
If now we endow $\mathcal{G}^\infty_{}(B_V)$ with its maximal operator space structure, then the mapping $\widehat{f}$ will be automatically completely bounded.
This approach, however, completely neglects to take into account the operator space structure in $W$ or $V$.
Our notions do precisely that.

As for the transference of approximation properties,  the space $\mathcal{G}^\infty(B_X)$ has the metric approximation property (resp. the approximation property) if and only if $X$ also does \cite[Prop. 5.7 and Thm. 5.4]{Mujica-linearization}. Analogously, we prove that $\Gcb(B_V)$ has the completely metric approximation property (resp. operator approximation property) if and only if $V$ does (Theorems \ref{thm-CMAP-transfers} and \ref{thm-OAP-transfers}).
The overall strategy of the proofs is similar to that of \cite{Mujica-linearization}, but the details are not simply a straightforward translation to the operator space setting.
The classical arguments in \cite{Mujica-linearization} crucially depend on endowing a space of bounded holomorphic functions with the topology of uniform convergence on compact subsets of the unit ball of a Banach space. The corresponding topology that we consider in the operator space setting (see Definition  \ref{defn-tau_0}) is inspired by Webster's approach to the Operator Approximation Property \cite{Webster}.

\section{Notation and preliminaries}\label{sec:prelim}

All Banach/operator spaces in this paper will be over the complex numbers.
The open unit ball of a Banach space $X$ will be denoted by $B_X$,
and the closed one by $\overline{B}_X$.
The dual of a Banach space $X$ is denoted $X'$.

\subsection{Polynomials and holomorphic functions}
We recall here some notation and basic facts about polynomials and holomorphic functions in Banach spaces.

Given Banach spaces $X$ and $Y$, a function $P:X\to Y$ is a {\em (continuous) $m$-homogeneous polynomial} if there exists a unique continuous symmetric $m$-linear mapping $\widecheck{P}:X\times\cdots\times X\to Y$ such that $P(x) = \widecheck{P}(x,\dots, x)$. We denote by $\mathcal P(^mX,Y)$ the set of all continuous $m$-homogeneous polynomials from $X$ to $Y$, which is a Banach space with the norm $\|P\|=\sup_{x\in B_X}\|P(x)\|$.

An $m$-homogeneous polynomial $P:X\to Y$ is said to be {\em of finite type} if there are linear functionals $x_1', \dots, x_N'\in X'$ and vectors $y_1,\dots, y_N\in Y$
such that $P(x)=\sum_{k=1}^N x_k'(x)^m y_k$.

A polynomial defined on a finite dimensional space is always of finite type. Thus, if $P$ is a polynomial and $T$ is a finite rank linear mapping then $P\circ T$ is a finite type polynomial. Note, also, that the image of a finite type polynomial is included in a finite dimensional space.

If $U\subset X$ is an open set, a mapping $f:U\to Y$ is said to be {\em holomorphic} if for every $x_0\in U$ there exists a sequence $(P^{[m]}f(x_0))$, with each $P^{[m]}f(x_0)$ a continuous $m$-homogeneous polynomial, such that the series
\[
f(x)=\sum_{m=0}^\infty P^{[m]}f(x_0)(x-x_0)
\] converges uniformly in some neighborhood of $x_0$ contained in $U$.

Equivalently, for every $x_0\in U$, the function $f$ is \textsl{Fr\'{e}chet differentiable} at $x_0$;  that is,
there exists a continuous linear mapping $df(x_0)$ from $X$ into $Y$ called the differential of $f$ at $x_0$, such that
$$
\lim_{h\to 0} \frac{f(x_0 +h)-f(x_0) -df(x_0)(h)}{\|h\|}=0.
$$

The set
$\mathcal H^\infty(B_X,Y)=\{f:B_X\to Y:\, f \textrm{ is holomorphic and bounded}\}$ is a Banach space if we endow it with the norm $\|f\|=\sup_{x\in B_X} \|f(x)\|$. For the case $Y=\mathbb C$, as usual, we  write $\mathcal H^\infty(B_X)$ instead of $\mathcal H^\infty(B_X, \mathbb C)$.

\subsection{Operator spaces}
We assume familiarity with the basic theory of operator spaces. The books \cite{Pisier-Operator-Space-Theory} and \cite{Effros-Ruan-book} are excellent references on this topic.

A linear map $Q : V \to W$ between operator spaces is called a \textit{complete 1-quotient}  if it is onto and the associated map from $V/\ker(Q)$ to $W$ is a completely isometric isomorphism.
In \cite[Sec. 2.4]{Pisier-Operator-Space-Theory}, these maps are referred to as complete metric surjections. 

 It is well known that for a given Banach space $X$, among all the  possible operator space structures which are compatible with the norm of $X$ there are a smallest one and a largest one, denoted $\MIN(X)$ and $\MAX(X)$ respectively (see \cite[Chap. 3]{Pisier-Operator-Space-Theory}).
 To be precise, $\MIN(X)$ (resp. $\MAX(X)$) is the unique operator structure on $X$ compatible with the norm of $X$ and such that for any operator space $W$ and any linear map $u : W \to X$ (resp. $v : X \to W$) we have
 $\|u : W \to \MIN(X)\|_{\cb} = \| u : W \to X \|$
 (resp. $\|v : \MAX(X) \to W\|_{\cb} = \| v : X \to W \|$).

We denote by $\mathcal{K}$ the space of compact operators on $\ell_2$, with its canonical operator space structure inherited from $\mathcal{B}(\ell_2)$. When $V$ is an operator space, $\mathcal{K}(V)$ denotes the operator space minimal tensor product of $\mathcal{K}$ and $V$, which we often understand as a space of infinite matrices with entries in $V$. 

An $n$-homogeneous polynomial $P:V\to W$ belongs to the Schur class $\mathcal P_s(^nV,W)$ \cite{Defant-Wiesner} if 
\[
\sup \left\{ \n{\big( P(x_{ij})\big)}_{M_m(W)} \;:\; m\in\N, (x_{ij}) \in B_{M_m(V)} \right\} <\infty,
\]
and this quantity is said to be the Schur norm of $P$. The identification
\[
M_m(\mathcal P_s(^nV,W))=\mathcal P_s(^nV,M_m(W))
\] gives $\mathcal P_s(^nV,W)$ the structure of an operator space.

\subsection{Matrix convexity}

The contents of this section are mainly from \cite{Effros-Webster}.
A \emph{matrix set $\mathbf{K} = (K_n)_n$ over a set $V$} is a sequence of subsets $K_n \subseteq M_n(V)$ for each $n\in\N$.
Given two matrix sets $\mathbf{K} = (K_n)_n$ and $\mathbf{L} = (L_n)_n$ over the same set $V$, we write $\mathbf{K} \subseteq \mathbf{L}$ when $K_n \subseteq L_n$ for all $n\in\N$.
When $V$ is an operator space, we say that  $\mathbf{K}$ is open (resp. closed) whenever each $K_n$ is open (resp. closed). By the closure of a matrix set $\mathbf{K} = (K_n)_n$ we mean taking the closure on each level: $\overline{\mathbf{K}} = ( \overline{K_n})_n$.
Note that a matrix set $\mathbf{K}$ over an operator space $V$ can be associated with a subset of $\K(V)$: using the standard identification of $M_n(V)$ with a subset of $\K(V)$ (filling up with zeros), each $K_n$ can be identified with a subset of $\K(V)$, and then we take the union of these. When we want to consider the closure in $\K(V)$ of the subset of $\K(V)$ which is associated with $\mathbf{K}$, we write $\overline{\mathbf{K}}^{\K(V)}$.

We say that a matrix set $\mathbf{K} = (K_n)$ over $V$ is \emph{absolutely matrix convex} if:
\begin{enumerate}[(i)]
	\item For all $x \in K_n$ and $y \in K_m$, $x \oplus y \in K_{m+n}$.
	\item For all $x \in K_n$, $\alpha \in M_{m,n}$, and $\beta \in M_{n,m}$ with $\n{\alpha}, \n{\beta} \le 1$,
	$\alpha x \beta \in K_m$.
\end{enumerate}

The main example of an absolutely matrix convex set of matrices over $V$ comes from the operator space structure: it follows immediately from Ruan's axioms that
both the \emph{open matrix unit ball of $V$} given by
$\mathbf{B}_V = \big( B_{M_n(V)} \big)_n$
and the 
\emph{closed matrix unit ball of $V$} given by
$\overline{\mathbf{B}}_V = \big( \overline{B}_{M_n(V)} \big)_n$
are absolutely matrix convex matrix sets over $V$.

Clearly, the intersection of a family of absolutely matrix convex matrix sets over $V$ (understood as intersecting on each matrix level) is again absolutely matrix convex. Therefore we can define the \emph{absolutely matrix convex hull} of a given matrix set $\mathbf{K}$ over $V$, denoted by $\amconv(\mathbf{K})$,
as the smallest absolutely matrix convex matrix set over $V$ that contains $\mathbf{K}$.
The following nice characterization of the absolutely matrix convex hull of a matrix set is due to B.E. Johnson \cite[Lemma 3.2]{Effros-Webster} (the slight restatement below is from \cite[Lemma 2.1]{CD-Oikhberg}).

\begin{lemma}\label{lemma-characterization-abs-mat-convex-hull}
	Given a matrix set $\mathbf{K}=(K_n)_n$ over $V$, the $n$-th level of $\amconv(\mathbf{K})$ is given by
	\begin{multline*}
	\bigg\{ \sum_{i=1}^m \alpha_i x_i \beta_i \; : \; x_i \in K_{k_i}, \alpha_i \in M_{n,k_i},\beta_i \in M_{k_i,n} \text{ such that } \bigg\| \sum_{i=1}^m \alpha_i\alpha_i^*  \bigg\| \le 1,  \bigg\|\sum_{i=1}^m \beta_i^*\beta_i \bigg\| \le 1  \bigg\}.
	\end{multline*}
\end{lemma}

We now state a matrix analogue of the fact that taking the convex hull of a set cannot increase the supremum of the norm of the vectors in the set.
For a matrix set $\mathbf{K}=(K_n)_n$ over an operator space $V$, we will denote
$\n{\mathbf{K}} = \sup\{ \n{x}_{M_n(V)} \;:\; n\in\N, x \in K_n \}$.
This is surely well-known, but we were unable to find a reference and have included the proof for completeness.

\begin{lemma}\label{lemma-norm-of-abs-mat-conv-hull}
For a matrix set $\mathbf{K}=(K_n)_n$ over an operator space $V$, $\n{\mathbf{K}} = \n{\amconv(\mathbf{K})}$.
\end{lemma}

\begin{proof}
Since $\mathbf{K} \subseteq \amconv(\mathbf{K})$, the $\le$ inequality is obvious.
By definition of $\n{\mathbf{K}}$, we clearly have $\mathbf{K} \subseteq \n{\mathbf{K}} \overline{\mathbf{B}}_V$.
Taking the absolutely matrix convex hull,
$\amconv(\mathbf{K}) \subseteq \n{\mathbf{K}} \overline{\mathbf{B}}_V$, finishing the proof.
\end{proof}

The following is a version of the Hahn-Banach theorem for matrix convexity \cite[Thm. 2.3]{CD-Oikhberg} (essentially proved in  \cite[Prop. 4.1]{Effros-Webster}).

\begin{theorem}\label{thm-Hahn-Banach-matrix-convexity}
	Let $\mathbf{K} = (K_n)_n$ be a closed absolutely convex set matrix set over $V$ and let $x_0 \in M_n(V) \setminus K_n$ for some $n\in\N$.
	Then there exists $f \in M_n(V')$ such that for all $m\in\N$ and all $x \in K_m$,
	$$
	\n{ \mpair{f}{x} }_{M_{mn}} \le 1 \quad \text{but} \quad \n{ \mpair{f}{x_0} }_{M_{n^2}} > 1,
	$$ 
	where
	$$
	\mpair{(f_{ij})_{ij}}{(x_{kl})_{kl}} = \big( f_{ij}(x_{kl}) \big)_{ijkl}.
	$$
\end{theorem}

To simplify some statements, whenever $\mathbf{K} = (K_n)_n$ is a matrix set we will write ``$x \in \mathbf{K}$'' as a shorthand for ``there exists $n\in\N$ such that $x\in K_n$''.

\section{Completely bounded holomorphic functions}

This section is devoted to present our object of study, the space of completely bounded holomorphic functions, and to give several of its basic properties as well as descriptions of the behaviour of its elements.

In the classical framework, bounded holomorphic functions defined on a open unit ball of a Banach space constitute a Banach space. Inspired by the definition of completely bounded linear mappings,  we propose, in our non commutative setting, a related notion in order to obtain an operator space of holomorphic functions.

\begin{definition}
Let $V$ and $W$ be operator spaces. A holomorphic function $f : B_V \to W$ is said to be a \emph{completely bounded holomorphic function} if
\[
\sup \left\{ \n{\big( f(x_{ij})\big)}_{M_m(W)} \;:\; m\in\N, (x_{ij}) \in B_{M_m(V)} \right\} <\infty 
\]
The space of all such completely bounded holomorphic functions is denoted by $\Hcb(B_V,W)$, and the supremum above is denoted by $\n{f}_{\Hcb(B_V,W)}$ or simply $\n{f}_{\Hcb}$. In the case $W=\C$, we write $\Hcb(B_V)$.
\end{definition}

Note that a more concise way of writing the norm of a completely bounded holomorphic function $f : B_V \to W$ is
\[
\n{f}_{\Hcb(B_V,W)} = \sup\big\{ \n{f_{n_x}(x)}_{M_{n_x}(W)} \;:\; x \in \mathbf{B}_V \big\},
\]
where for each $x \in \mathbf{B}_V$ the number $n_x\in\N$ is the matrix level to which $x$ belongs (that is, $x \in M_{n_x}(V)$).
Therefore, the map $\Phi_W : \Hcb(B_V,W) \to \ell_\infty\{ M_{n_x}(W) \,:\, x \in \mathbf{B}_V \}$  given by $\Phi_W(f) = \big( f_{n_x}(x) \big)_{x\in \mathbf{B}_V}$ is an isometry.
Moreover, since $\ell_\infty\{ M_{n_x}(W) \,:\, x \in \mathbf{B}_V \}$ has a natural operator space structure, identifying $\Hcb(B_V,W)$ with its image under $\Phi_W$ we can define norms for matrices of elements of $\Hcb(B_V,W)$. Specifically, for $(f_{k\ell}) \in M_n( \Hcb(B_V,W) )$ we have that
\begin{equation}\label{eqn-oss-for-Hcb}
\n{ (f_{k\ell}) }_{M_n(\Hcb(B_V,W))} = 
\sup \left\{ \n{\big( f_{k\ell}(x_{ij})\big)}_{M_{mn}(W)} \;:\; m\in\N, (x_{ij}) \in B_{M_m(V)} \right\}.
\end{equation}

Now we see that with this structure, $\Hcb(B_V,W)$ is an operator space.

\begin{lemma}\label{lemma-Hcb-Banach}
For any operator spaces $V$ and $W$, the space
$\Hcb(B_V,W)$ is an operator space.    
\end{lemma}

\begin{proof}
It is clear that $\n{\cdot}_{M_n(\Hcb(B_V,W))}$ are indeed norms satisfying Ruan's conditions, so we only need to show completeness. Let $(f^{[n]})$ be a Cauchy sequence in $\Hcb(B_V,W)$, in particular it is bounded: there exists $C \ge 0$ such that for each $n\in\N$ we have $\n{f^{[n]}}_{\Hcb} \le C$. 
Note that $(f^{[n]})$ is Cauchy in $\mathcal{H}^\infty(B_V,W)$, so it has a limit in there which is a holomorphic function $f : B_V \to W$ and in particular $(f^{[n]})$ converges to $f$ pointwise.
For each fixed $x \in \mathbf{B}_V$, the sequence $(f^{[n]}_{n_x}(x))$ is bounded by $C$ in $M_{n_x}(W)$ and converges pointwise to $f_{n_x}(x)$, hence $\n{f_{n_x}(x)}_{M_{n_x}(W)} \le C$. We conclude that $f \in \Hcb(B_V,W)$ and a standard argument shows that $(f^{[n]})$ converges to $f$ in $\Hcb(B_V,W)$.
\end{proof}

Note that our proposed definition forces the functions in $\Hcb(B_V,W)$ to be null at 0:

\begin{remark}\label{remark-0-at-0}
Let $V$ and $W$ be operator spaces.
If  $f \in \Hcb(B_V,W)$, then $f(0)=0$.
Indeed, suppose on the contrary that $f(0)=w_0\not=0$. Since the matrix of all $0$'s is in $B_{M_m(V)}$ and the norm of the matrix of all $w_0$'s goes to infinity, $f$ is not completely bounded.
\end{remark}

We now present our first example of completely bounded holomorphic functions:

\begin{remark}\label{remark-linear-cb-is-holomorphic-cb}
It is clear that if $T\in\CB(V,W)$, then its restriction to $B_V$ is a completely bounded holomorphic function. Furthermore, this restriction process yields a completely isometric embedding $\CB(V,W) \to \Hcb(B_V,W)$. 
Nonlinear completely bounded holomorphic functions do exist: for example, the next proposition implies that all finite type polynomials on $V$ (that take the value 0 at 0) belong to $\Hcb(B_V)$ and hence the same holds for approximable polynomials (i. e. those belonging to the closure, in the supremum norm, of finite type polynomials). 
\end{remark}

The space $\mathcal{H}^\infty(B_X)$ is clearly a Banach algebra.
We now show that in the operator space setting $\Hcb(B_V)$ not only is also a Banach algebra, but in fact is a quantized Banach algebra in the sense of \cite[Chap. 16]{Effros-Ruan-book}.
Recall that a \emph{completely contractive Banach algebra} $\mathcal A$ is an operator space which is an associative algebra such that the multiplication $\mathcal A\times \mathcal A\to\mathcal A$ is jointly completely contractive.
The proof of the proposition below is  related to that of \cite[Prop. 3.2]{braga2024small}.

\begin{proposition}\label{proposition-Hinfty-is-algebra}
For any operator space $V$, the space $\Hcb(B_V)$ is a completely contractive Banach algebra.
More generally, for any operator algebra $\mathcal{A} \subseteq \mathcal{B}(H)$ (with the induced operator space structure) the space $\Hcb(B_V,\mathcal{A})$ is a completely contractive Banach algebra.
\end{proposition}

\begin{proof}
To illustrate the idea of the proof, let us first just show that $\Hcb(B_V)$ is an algebra.
This is essentially an immediate consequence of a theorem of Schur \cite[Satz III]{Schur} (see also \cite{Ong}) which states that the norm of the Schur product of two matrices in $M_n$ (that is, their entrywise product) is less than or equal to the product of the individual norms.
Therefore, if $f,g \in \Hcb(B_V)$ and $x=(x_{ij})\in B_{M_n(V)}$,
\[
\n{\big( (fg)(x_{ij}) \big)}_{M_n} = \n{\big( f(x_{ij}) g(x_{ij}) \big)}_{M_n} \le    \n{\big( f(x_{ij}) \big)}_{M_n} \n{\big( g(x_{ij}) \big)}_{M_n} 
\le \n{f}_{\Hcb} \n{g}_{\Hcb},
\]
showing that $fg \in \Hcb(B_V)$ and $\n{fg}_{\Hcb} 
 \le \n{f}_{\Hcb} \n{g}_{\Hcb}$.

For the general operator-valued case (and the complete contractivity),
for $n\in\N$ we denote by $\square : M_n(\mathcal{B}(H)) \times M_n(\mathcal{B}(H)) \to M_n(\mathcal{B}(H))$ the Schur product, that is, for any $(A_{ij}),(B_{ij}) \in M_n(\mathcal{B}(H))$ define $(A_{ij}) \square (B_{ij}) = (A_{ij}B_{ij})$.
It is shown in  \cite[Thm. 2.3]{Christensen} that $\square$ is a completely bounded bilinear map of norm 1, and therefore it is jointly completely contractive \cite[Eqn. (9.1.11)]{Effros-Ruan-book}.
\footnote{
The reader should be warned that some older references such as \cite{Effros-Ruan-book} use the terminology \emph{completely bounded} (resp. \emph{multiplicatively bounded}) for the multilinear mappings that are nowadays often called \emph{jointly completely bounded} (resp. \emph{completely bounded}).
}
Hence, if $f=(f_{kl}) \in M_p(\Hcb(B_V,\mathcal{A}))$, $g=(g_{rs}) \in M_q(\Hcb(B_V,\mathcal{A}))$  and $x=(x_{ij})\in B_{M_n(V)}$
\begin{multline*}
\n{ \big( f_{kl}(x_{ij}) g_{rs}(x_{ij}) \big) }_{ M_{pqn(\mathcal{B}(H))} } = \n{ \Big( \big( f_{kl}(x_{ij}) \big) \square \big( g_{rs}(x_{ij}) \big) \Big) }_{ M_{pq}(M_n(\mathcal{B}(H)))} \\
\le \n{\square}_{\jcb} \n{\Big( \big( f_{kl}(x_{ij}) \big)}_{ M_{p}(M_n(\mathcal{B}(H)))} \n{ \big( g_{rs}(x_{ij}) \big)}_{ M_{q}(M_n(\mathcal{B}(H)))} \\
\le \n{ (f_{kl}) }_{M_p(\Hcb(B_V,\mathcal{A}))} \n{ (g_{rs}) }_{M_q(\Hcb(B_V,\mathcal{A}))},
\end{multline*}
which shows that
\[
 \n{ (f_{kl}g_{rs}) }_{M_{pq}(\Hcb(B_V,\mathcal{A}))} \le \n{ (f_{kl}) }_{M_p(\Hcb(B_V,\mathcal{A}))} \n{ (g_{rs}) }_{M_q(\Hcb(B_V,\mathcal{A}))},
\]
which is precisely what it means for $\Hcb(B_V,\mathcal{A})$ to be a completely contractive Banach algebra.
\end{proof}

Note that the Banach algebras from the preceding proposition are not unital, since constant functions do not belong to $\Hcb (B_V)$.

We now present examples of nonlinear completely bounded holomorphic functions, namely certain polynomials.

\begin{example}
Directly from our definition, the Schur polynomials $V \to W$ from \cite{Defant-Wiesner} are in $\Hcb(V,W)$.
Thus, it follows from \cite[Prop. 9.3]{Defant-Wiesner} that whenever $V$ is a maximal operator space we have $\mathcal{P}(^m V,W) \subset \Hcb(V,W)$.
Similarly, since the Kronecker polynomials from \cite{Defant-Wiesner} are Schur \cite[Cor. 5.3]{Defant-Wiesner}, it follows from the discussion before \cite[Thm. 7.1]{Defant-Wiesner} that the completely bounded polynomials from \cite{Dineen-Radu-polynomials} are also in $\Hcb(V,W)$. In particular, it follows from \cite[Prop. 4.1]{Dineen-Radu-polynomials} that if
 $(e_j)_{j=1}^\infty$ is an orthonormal basis for $OH$, $W$ is an operator space, and $(w_j)_{j=1}^\infty$ is a norm null sequence in $W$, then for any natural number $m \ge 2$ the polynomial $P : OH \to W$ given by
\[
P\left( \sum_{n=1}^\infty x_j e_j \right) = \sum_{n=1}^\infty x_j^m w_j.
\]
is in $\Hcb(OH,W)$. See also \cite{Alaminos-Extremera-Villena} for a discussion of certain completely bounded polynomials (in the sense of \cite{Dineen-Radu-polynomials}) on Fourier algebras of locally compact groups.
\end{example}

A non-polynomial example is given next.

\begin{proposition}\label{proposition-phi-over-1-minus-phi}
Let $V$ be an operator space.
For each $\varphi \in B_{V'}$, the function $f : B_V \to \C$ given by $f(x) = \varphi(x)/(1-\varphi(x))$ is in $\Hcb(B_V)$. Moreover, $\n{f}_{\Hcb} \le \n{\varphi}/(1-\n{\varphi})$.   
\end{proposition}

\begin{proof}
Let $\varphi \in V'$ with $\n{\varphi}=\n{\varphi}_{\cb} < 1$.
We know that $\varphi \in \Hcb(B_V)$, and by Proposition \ref{proposition-Hinfty-is-algebra} for each $n\in\N$ the function $\varphi^n$ is also in $\Hcb(B_V)$ with $\n{\varphi^n}_{\Hcb}  \le \n{\varphi}_{\cb}^n$.
Therefore $\sum_{n=1}^\infty \varphi^n$ converges in $\Hcb(B_V)$ to a function $f$.
Note that for each $x \in B_V$, the series $\sum_{n=1}^\infty \varphi(x)^n$ converges to $\varphi(x)/(1-\varphi(x))$, yielding the result. The norm estimate is immediate from the argument.
\end{proof}

In particular, functions of the form $z/(1-az)$ with $a\in\D$ belong to $\Hcb(\D)$. The next result gives even more examples of analytic functions in $\Hcb(\D)$.

\begin{proposition}\label{proposition-finite-Blaschke-products}
\begin{enumerate}[(a)] 
\item If $f \in \Hcb(\D)$ and $a\in \D$, then $g(z)=f(z)/(1-az)$ is also in $\Hcb(\D)$. Moreover, $\n{g}_{\Hcb} \le \n{f}_{\Hcb}/(1-|a|)$.
\item Any finite Blaschke product vanishing at $0$ belongs to $\Hcb(\D)$.
\item The set of functions in the Wiener algebra $\mathcal{W}(\D)$ which vanish at 0 is a subset of $\Hcb(\D)$.
\end{enumerate}    
\end{proposition}

\begin{proof}
Given $f \in \Hcb(\D)$, note that
\[
\frac{f(z)}{1-az} = f(z) + \sum_{n=1}^\infty f(z) (az)^n.
\]
The series converges absolutely in $\Hcb(\D)$ by using Proposition \ref{proposition-Hinfty-is-algebra} as we did in the proof of Proposition \ref{proposition-phi-over-1-minus-phi}, yielding the desired result. Once again, the norm estimate is immediate from the argument.

Now consider a finite Blaschke product of the form 
\[
\mathfrak{B}(z) = c z^m \prod_{j=1}^n \frac{z-a_j}{1 - \overline{a_j}z} 
\]
where $m \ge 1$, $c\in\T$ and $a_1 ,\dotsc, a_n \in \D$.
Note that $P(z) = c z^m \prod_{j=1}^n (z-a_j)$ is a polynomial vanishing at $0$, which therefore belongs to $\Hcb(\D)$. Repeated applications of the first part yield that $\mathfrak{B}(z)$ belongs to  $\Hcb(\D)$ as well.

Finally, suppose that $f \in \mathcal{W}(\D)$ satisfies $f(0)=0$. Then $f$ has a Taylor series $f(z) = \sum_{n=1}^\infty a_n z^n$ with $\sum_{n=1}^\infty |a_n| < \infty$, so by the same arguments as above we conclude $f \in \Hcb(\D)$.
\end{proof}

\begin{question}
Note that in the vector-valued case, we have plenty of examples of holomorphic functions in $\mathcal{H}^\infty(B_V,W)$ which are not in $\Hcb(B_V,W)$: simply consider a linear map $V \to W$ which is bounded but not completely bounded.
In the scalar-valued case, we have not been able to produce analogous examples. In particular, we ask: does every $f \in \mathcal{H}^\infty(\D)$ with $f(0)=0$ belong to $\Hcb(\D)$?
\end{question}

We now observe that amplifications of holomorphic functions are also holomorphic.

\begin{lemma}\label{lemma-amplifications-are-holomorphic}
Let $V$ and $W$ be operator spaces, and let
 $f_{ij}: B_V \to W$ be holomorphic functions for $1 \le i,j \le n$.
 \begin{enumerate}[(a)]
 \item The function $g : B_{V} \to M_n(W)$ given by $g(x) = \big(f_{ij}(x)\big)$ is holomorphic. Moreover, for each $x \in B_{V}$ we have $dg(x) = \big( df_{ij}(x)\big)$.
\item The function $f : B_{M_n(V)} \to M_n(W)$ given by $f(x_{ij}) = \big(f_{ij}(x_{ij})\big)$ is holomorphic. Moreover, for each $x=(x_{ij}) \in B_{M_n(V)}$ we have $df(x) = \big( df_{ij}(x_{ij})\big)$. In particular, if $f : B_V \to W$ is a holomorphic function, then for every $n$ its amplification $f_n$ is holomorphic from $B_{M_n(V)}$ to $M_n(W)$. Also, for each $x=(x_{ij}) \in B_{M_n(V)}$ we have $df_n(x) = \big( df(x_{ij})\big)$.
 \end{enumerate}
\end{lemma}

\begin{proof}
For part $(a)$, fix $x \in B_V$ and consider nonzero
 $h \in V$ of sufficiently small norm so that $x+h$ still belongs to $B_{V}$. Thus,
\[
\frac{\n{ \big( f_{ij}(x+h) - f_{ij}(x) -  df_{ij}(x)h \big) }_{M_n(W)}}{\n{h}_{V}}
\le \sum_{i,j=1}^n \frac{\n{ f_{ij}(x+h) - f_{ij}(x) -  df_{ij}(x)h  }_{W}}{\n{h}_{V}}
\]
which implies the desired conclusion.

For part $(b)$, now fix $x=(x_{ij}) \in B_{M_n(V)}$.
For any nonzero $h=(h_{ij}) \in M_n(V)$ of sufficiently small norm so that $x+h$ still belongs to $B_{M_n(V)}$, we have
\begin{multline*}
\frac{\n{ f(x+h) - f(x) - \big( df_{ij}(x_{ij})h_{ij} \big) }_{M_n(W)}}{\n{h}_{M_n(V)}}
=
\frac{\n{ \big( f_{ij}(x_{ij}+h_{ij}) - f_{ij}(x_{ij}) -  df(x_{ij})h_{ij} \big) }_{M_n(W)}}{\n{h}_{M_n(V)}} \\
\le \sum_{i,j=1}^n \frac{\n{ f_{ij}(x_{ij}+h_{ij}) - f_{ij}(x_{ij}) -  df_{ij}(x_{ij})h_{ij}  }_{W}}{\n{h}_{M_n(V)}} 
\le \sum_{i,j|\, h_{ij}\not=0} \frac{\n{ f_{ij}(x_{ij}+h_{ij}) - f_{ij}(x_{ij}) -  df_{ij}(x_{ij})h_{ij}  }_{W}}{\n{h_{ij}}_{V}}
\end{multline*}
 which implies the desired conclusion.
\end{proof}

\begin{remark}\label{remark-Hcb-in-terms-of-H}
Note that from the previous lemma, whenever $(f_{ij}) \in M_n(\Hcb(B_V,W))$ we have
\[
\n{(f_{ij})}_{M_n(\Hcb(B_V,W))} = \sup_m \n{\big((f_{ij})_m\big)}_{\mathcal{H}^\infty(B_{M_m(V)};M_{mn}(W))}.
\] 
The equality of the norms is clear from \eqref{eqn-oss-for-Hcb}, the role of the previous lemma is to justify that the norms of the right hand side make sense.
\end{remark}

Note that from \eqref{eqn-oss-for-Hcb}, for any $(x_{kl}) \in B_{M_m(V)}$ and any $(f_{ij}) \in M_n(\Hcb(B_V,W))$ we have 
$\|\big((f_{ij}(x_{kl})\big)\|_{M_{nm}(W)} \le  \n{(f_{ij})}_{M_n(\Hcb(B_V,W))}$. In the next corollary we show a refinement of this inequality.

\begin{corollary}\label{corollary-Schwarz}
Let $V$ and $W$ be operator spaces.
\begin{enumerate}[(a)]
\item For any 
 $f \in \Hcb(B_V,W)$ and $x \in B_{M_n(V)}$, we have $\n{f_n(x)}_{M_n(W)} \le \n{f}_{\Hcb} \n{x}_{M_n(V)}$.
% \item For any $x = (x_{ij}) \in B_{\K(V)}$ we have $f_\infty(x):=\big( f(x_{ij}) \big) \in \n{f}_{\Hcb} \overline{B}_{\K(W)}$.
\item For any $(x_{kl}) \in B_{M_m(V)}$ and any $(f_{ij}) \in M_n(\Hcb(B_V,W))$ we have 
$$\|\big((f_{ij}(x_{kl})\big)\|_{M_{nm}(W)} \le  \n{(f_{ij})}_{M_n(\Hcb(B_V,W))} \n{(x_{kl})}_{M_m(V)}.$$
\item Given 
 $f \in \Hcb(B_V,W)$, defining for each $x = (x_{ij}) \in B_{\K(V)}$ the function $f_\infty(x):=\big( f(x_{ij}) \big) $ results that $f_\infty\in \mathcal H^\infty(B_{\K(V)},\K(W))$ with $\|f_\infty\|_{\mathcal{H}^\infty}=\n{f}_{\Hcb}$.
Furthermore, $f_\infty\in\Hcb(B_{\K(V)},\K(W))$ with $\|f_\infty\|_{\Hcb}=\n{f}_{\Hcb}$.
\end{enumerate}
\end{corollary}

\begin{proof}
By Lemma \ref{lemma-amplifications-are-holomorphic} the amplification $f_n : B_{M_n(V)} \to M_n(W)$ is holomorphic, and by Remark \ref{remark-0-at-0} it sends $0$ to $0$. Part $(a)$ now follows from Schwarz's lemma \cite[Thm. 7.19]{Mujica-book}.
Part $(b)$ follows by the exact same argument applied to the mapping $F:B_{M_m(V)}\to M_{nm}(W)$ given by $F((x_{kl}))=(f_{ij}(x_{kl})) $, which is holomorphic by Lemma \ref{lemma-amplifications-are-holomorphic}.

To show $(c)$, for any $n\in\N$ let us denote by $\rho_n:B_{\K(V)}\to B_{M_n(V)}$ and $\iota_n: M_n(W)\to \K(W)$ the natural restriction and injection. The sequence $(\iota_n\circ f_n\circ\rho_n)$ belongs to $\mathcal H^\infty(B_{\K(V)},\K(W))$ and it is bounded by 
$\n{f}_{\Hcb}$. From condition $f(0)=0$ and part $(a)$ we derive that, for $n<m$ and $x\in B_{\K(V)}$, 
\[
\|(\iota_m\circ f_m\circ\rho_m-\iota_n\circ f_n\circ\rho_n)(x)\| = \|\iota_m\circ f_m(\rho_m(x)-\rho_n(x))\|\le  \n{f}_{\Hcb} \n{(\rho_m-\rho_n)(x)}_{M_m(V)}.
\]
As a consequence, $(\iota_n\circ f_n\circ\rho_n)$ is Cauchy for the topology of uniform convergence on compact sets of $B_{\K(V)}$. Since this sequence is pointwise convergent to $f_\infty$, applying a vector-valued version of the Weierstrass-type  theorem \cite[Th. 2.13]{defant2019dirichlet} we get that $f_\infty\in \mathcal H^\infty(B_{\K(V)},\K(W))$ with $\|f_\infty\|_{\mathcal{H}^\infty}\le\n{f}_{\Hcb}$. The other inequality is clear by definition, so the norms are equal.

Now, from \cite[Cor. 5.10 and Eqn. (5.16)']{Pisier-Operator-Space-Theory} and the associativity of the Haagerup tensor product,
if $R$ and $C$ denote the row and column Hilbertian operator spaces, respectively, and $R_n$, $C_n$ their $n$-dimensional versions, 
\[
M_n(\K(V)) = C_n \widehat{\otimes}_h ( C \widehat{\otimes}_h V \widehat{\otimes}_h R) \widehat{\otimes}_h R_n =  (C_n \widehat{\otimes}_h  C) \widehat{\otimes}_h V \widehat{\otimes}_h (R \widehat{\otimes}_h R_n ) = C \widehat{\otimes}_h V \widehat{\otimes}_h R = \K(V),
\]
where the equal signs mean completely isometric isomorphisms, and similarly $M_n(\K(W)) = \K(W)$. Under these identifications, the amplification $(f_\infty)_n : B_{M_n(\K(V))} \to M_n(\K(W))$ corresponds to $f_\infty : B_{\K(V)} \to \K(W)$, from where it follows that $ \|f_\infty\|_{\Hcb}=\|f_\infty\|_{\mathcal{H}^\infty}$.

\end{proof}

Note that Corollary \ref{corollary-Schwarz} $(c)$ above is a version for holomorphic functions of the following well known fact \cite[p. 32]{Pisier-Operator-Space-Theory}: a linear map $T : V \to W$ is completely bounded if and only if $T_\infty : \K(V) \to \K(W)$ is bounded, and in this case $\n{T}_{\cb} = \n{T_\infty} = \n{T_\infty}_{\cb}$.

\begin{remark}\label{rmk:polynomials}
Note that if $f=P\in \Hcb(B_V,W)$ is an $m$-homogeneous polynomial then $P_\infty\in \mathcal H^\infty(B_{\K(V)},\K(W))$ is also an $m$-homogeneous polynomial.
Indeed, this holds because $m$-homogeneous polynomials are the restriction to the diagonal of $m$-linear mappings and it is clear that the $\infty$-amplification of an $m$-linear mapping is again $m$-linear.
\end{remark}

We now prove that for a completely bounded holomorphic function, the polynomials appearing in its Taylor series centered at 0 are also completely bounded holomorphic functions.

\begin{lemma}\label{lemma-Taylor-polynomials-at-0-are-Schur}
Let $V$ and $W$ be operator spaces.
If $f : B_V \to W$ is a completely bounded holomorphic function, then each of the polynomials in the Taylor expansion of $f$ at $0$ belongs to $\Hcb(B_V,W)$  (that is, they are Schur polynomials). Moreover, their norms are bounded by $ \n{f}_{\Hcb}$.
\end{lemma}

\begin{proof}
Let $f(x) = \sum_{n=1}^\infty P^{[m]}(x)$ be the Taylor series of $f$ at $0$, where each $P^{[m]} \in \mathcal{P}(^mV,W)$ and let $r<1$ such that the series converges uniformly on $rB_V$.
From \cite{Defant-Wiesner}, for each $n\in \N$ the $n$-th amplification of $P^{[m]}$ is a polynomial in  $\mathcal{P}(^mM_n(V),M_n(W))$.
Since $r B_{M_n(V)} \subset M_n(rB_V)$ (that is, the norm of a matrix in $M_n(V)$ dominates the norms of each of its entries), for each $(x_{ij})$ in $r B_{M_n(V)}$ we have pointwise
\[
f_n(x_{ij}) = \big(f(x_{ij})\big) = \sum_{n=1}^\infty P^{[m]}_n((x_{ij})).
\]
Moreover, by the triangle inequality the convergence is uniform on $r B_{M_n(V)}$.
This implies that $P^{[m]}_n$ is the $m$-th Taylor polynomial for $f_n$ at $0$, so by the Cauchy inequality \cite[Cor. 7.4]{Mujica-book} for every $\rho\in(0,1)$ we have
\[
\n{P^{[m]}_n}_{\mathcal{H}^\infty(\rho B_{M_n(V)},M_n(W))} \le \rho^{-m} \n{f}_{\Hcb(B_V,W)}.
\]
Taking the supremum over all such $\rho$, the conclusion follows.    
\end{proof}

Note that from the previous result, in particular the derivative at 0 of a completely bounded holomorphic function is a completely bounded linear mapping.
We do not know whether this is also true at points other than $0$. That is, if $f : B_V \to W$ is a completely bounded holomorphic function, are its derivatives at all points completely bounded?

\begin{corollary}\label{corollary-biholomorphic-balls-implies-isomorphism}
Let $V$ and $W$ be operator spaces.
If there is a bijection $f : B_V \to B_W$ such that both $f$  and $f^{-1}$ are completely bounded holomorphic functions, then $V$ and $W$ are completely isomorphic.   
\end{corollary}

\begin{proof}
Take the derivative at $0$ for $Id = f^{-1}f : B_V \to V$.
Noting that $f(0)=0$ by Remark \ref{remark-0-at-0}, Lemma \ref{lemma-Taylor-polynomials-at-0-are-Schur} yields the conclusion.
\end{proof}

In fact more is true, an operator space version of the classical fact that Banach spaces with biholomorphic unit balls are isometric \cite{Kaup-Upmeier}.

\begin{proposition}\label{proposition-biholomorphic-balls-implies-isometry}
Let $V$ and $W$ be operator spaces.
If there is a bijection $f : B_V \to B_W$ such that both $f$  and $f^{-1}$ are completely bounded holomorphic functions, then $V$ and $W$ are completely isometric. Moreover, $f$ is the restriction to $B_V$ of a complete isometry between $V$ and $W$.       
\end{proposition}

\begin{proof}
For each $n\in\N$, the amplifications $f_n : B_{M_n(V)} \to B_{M_n(W)}$ and $(f^{-1})_n : B_{M_n(W)} \to B_{M_n(V)}$ are clearly injective. Since they are inverses of each other, they must be surjective and hence bijections.
Once again noting that $f(0)=0$ by Remark \ref{remark-0-at-0}, applying \cite[Lemma 2.1]{Arazy} we get that for each $n\in\N$ there is a linear isometry $T^{[n]} : M_n(V) \to M_n(W)$ whose restriction to $B_{M_n(V)}$ is precisely $f_n$. It follows that $T^{[n]}$ must be the $n$-th amplification of $T^{[1]}$, and the desired conclusion follows.
\end{proof}

\begin{remark}
Note that the proof above is easier than in the classical case because we have the condition $f(0)=0$; it was already pointed out in \cite{Arazy} that this condition allows for simplified arguments.
\end{remark}

\section{The predual of $\Hcb(B_V)$: Dixmier-Ng-Kaijser approach}\label{Sect:Dixmier-Ng-Kaijser}

One of the various classical ways to see that $\mathcal{H}^\infty(U)$ is a dual space is by using the Dixmier-Ng theorem, see e.g. \cite[Proof of Thm. 2.1]{Mujica-linearization} (and \cite[p. 7]{aron2023linearization} for a similar result for holomorphic Lipschitz functions). Moving to the operator space framework, note that if $X'=V$ as Banach spaces and $V$ is an operator space, one can naturally endow $X$ with an operator space structure since $X$ sits canonically inside $X''=V'$ and the latter has a natural operator space structure. The question now is whether in this way $X$ becomes an operator space predual of $V$, and this was explicitly asked by Blecher. Le Merdy \cite{LeMerdy-duality} showed that the answer in general is no, and provided necessary and sufficient conditions for a positive answer.
A version of Kaijser \cite{Kaijser} of the Dixmier-Ng theorem  combined with  Le Merdy's result  about conditions for operator space preduals, yields the following:

\begin{theorem}\label{thm-Kaijser}
Let $V$ be an operator space, and let $Y \subset V'$ be a set of continuous linear functionals on $V$ such that:
\begin{enumerate}[(a)]
\item $Y$ separates points of $V$.
\item For each $n\in\N$, the closed unit ball of $M_n(V)$ is compact for the entrywise $\sigma(V,Y)$-topology.
\end{enumerate}
Then $V$ is a dual operator space, and a predual for $V$ is the closure in $V'$ of the linear span of the set $Y$.    
\end{theorem}

\begin{proof}
Due to $(a)$ and condition $(b)$ for $n=1$, it follows from \cite{Kaijser}  that $V$ is a dual Banach space, with Banach space predual $W = \overline{\spa} Y \subset V'$. 
Moreover, note that on $B_V$ the $\sigma(V,Y)$ and $\sigma(V,W)$ topologies are the same: the identity map on $B_V$ is clearly a continuous bijection from the $\sigma(V,W)$ to the $\sigma(V,Y)$ topology, and both of these topologies are compact and Hausdorff, so they must coincide.
By \cite[Prop. 3.1]{LeMerdy-duality}, in order to obtain the conclusion it suffices to show that for each $n\in\N$ the set $B_{M_n(V)}$ is closed in the entrywise $\sigma(V,W)$-topology. This follows from the assumption $(b)$ and the aforementioned coincidence of the $\sigma(V,Y)$ and $\sigma(V,W)$ topologies on $B_V$.
\end{proof}

Before stating our duality theorem, we remark that from the definition of the norm on $\Hcb(B_V)$ it is clear that for any $x\in B_V$ the evaluation functional $\delta_V(x)$ belongs to $\Hcb(B_V)'$ and   $\|\delta_V(x)\|=\|x\|$. Indeed, by Corollary \ref{corollary-Schwarz} $(a)$ the norm of  $\delta_V(x)$ is at most $\|x\|$. The other inequality follows taking $x'\in S_{V'}$  such that $x'(x)=\|x\|$, so $\delta_V(x)(x')=\|x\|$. Note that, at first sight, this appears to be in sharp contrast with the classical case \cite{Mujica-linearization}: for every $x \in B_V$ one has $\n{\delta_V(x)}_{\mathcal H^\infty(B_V)'}=1$.
This is due to the fact that functions in $\Hcb(B_V)$ have to take the value $0$ at $0$, whereas those in $\mathcal{H}^\infty(B_V)$ do not have this restriction. 

We now apply the previous theorem to produce an operator space predual of $\Hcb(B_V)$.

\begin{theorem}
For any operator space $V$ the space $\Hcb(B_V)$ is a dual operator space, and in fact $\Gcb(B_V) = \overline{\spa}\{ \delta_V(x) \}_{x \in B_V} \subseteq \Hcb(B_V)'$ is an operator space predual of $\Hcb(B_V)$.
\end{theorem}

\begin{proof}
Let $Y =\{ \delta_V(x) \}_{x \in B_V} \subseteq  \Hcb(B_V)'$. It is clear that it separates points of $\Hcb(B_V)$, and the $\sigma(\Hcb(B_V), Y)$-topology is the topology of pointwise convergence.
Let $(f^{[\alpha]})_{\alpha\in A}$ be a net in the closed unit ball of $\Hcb(B_V)$. In particular it is in the closed unit ball of  $\mathcal{H}^\infty(B_V)$, so it has a subnet $(f^{[\gamma]})_{\gamma\in\Gamma}$ that converges pointwise to an $f$ in the closed unit ball of $\mathcal{H}^\infty(B_V)$. It follows from the pointwise convergence that $\n{f_{n_x}(x)}_{M_{n_x}} \le 1$ for any $x \in \mathbf{B}_V$, and therefore $f$ belongs to the closed unit ball of $\Hcb(B_V)$.
This shows that condition $(b)$ in Theorem \ref{thm-Kaijser} holds for $n=1$.

For $n>1$, a similar argument shows that given a net in $B_{M_n(\Hcb(B_V))}$ we can find a subnet which converges pointwise.
By \eqref{eqn-oss-for-Hcb}, the pointwise limit of this subnet will belong to $B_{M_n(\Hcb(B_V))}$, which shows that $B_{M_n(\Hcb(B_V))}$ is compact for the entrywise $\sigma(\Hcb(B_V), Y)$-topology.
Thus, by Theorem \ref{thm-Kaijser},  $\Gcb(B_V)$ is an operator space predual of $\Hcb(B_V)$.
\end{proof}

In the theorem above we have found an operator space predual of $\Hcb(B_V)$ using very abstract tools, but in the next section we will show an alternative description of its operator space structure that is perhaps more explicit (see Propositions \ref{proposition-ball-is-convex-hull-of-image-of-delta} and \ref{prop-description-oss-Gcb}).

Now that we have found a predual for $\Hcb(B_V)$ we show that, as in the typical linearization procedures for Banach spaces, the evaluation map $\delta_V : B_V \to \Gcb(B_V)$ belongs to the class of functions being linearized: it is a completely bounded holomorphic function (and it has norm one).

\begin{proposition}\label{prop-delta-is-cb-holomorphic}
Let $V$ be an operator space.
The evaluation map $\delta_V : B_V \to \Gcb(B_V)$ given by $x \mapsto \delta_V(x)$
is a completely bounded holomorphic function, and furthermore the norm of $\delta_V$ in $\Hcb(B_V; \Gcb(B_V))$ is 1.
More precisely, $\delta_V$ is completely isometric in the sense that for all $(x_{ij})\in M_n(V)$ we have
\[
\n{(\delta_Vx_{ij})}_{M_n(\Gcb(B_V))} = \n{(x_{ij})}_{M_n(V)}.
\]
\end{proposition}

\begin{proof}
The fact that $\delta_V$ is holomorphic follows analogously as in \cite[Prop. 2.5(a)]{aron2023linearization}: $\delta_V$ is weakly holomorphic because for every $f \in \Gcb(B_V)' = \Hcb(B_V)$
we have $f \circ \delta_V = f$ is holomorphic, and hence $\delta_V$ is holomorphic \cite[Thm. 8.12]{Mujica-book}.
Consider now a matrix $(x_{ij}) \in B_{M_n(V)}$.
Then, from the duality $\Gcb(B_V)' = \Hcb(B_V)$,
\begin{multline*}
\n{(\delta_Vx_{ij})}_{M_n(\Gcb(B_V))} = 
\sup \left\{ \n{( \delta_V(x_{ij})f_{kl} )}_{M_{nm}(W)} \;:\; (f_{kl}) \in B_{M_m(\Hcb(B_V))} \right\}
= \\
\sup \left\{ \n{( f_{kl}(x_{ij}) )}_{M_{nm}(W)} \;:\; (f_{kl}) \in B_{M_m(\Hcb(B_V))} \right\}.
\end{multline*}
Note that above we can understand $f = (f_{kl}) \in B_{M_m(\Hcb(B_V))}$ as a single holomorphic function $f : B_V \to M_m(W)$ by the same argument as in Lemma \ref{lemma-amplifications-are-holomorphic}. Therefore, by that same lemma, its $n$-th amplification $f_n$ is a holomorphic function $B_{M_n(V)} \to M_{nm}(W)$. Since $f_n$ takes the value $0$ at $0$, by the same argument as in the proof of Corollary \ref{corollary-Schwarz} it follows from Schwarz's lemma \cite[Thm. 7.19]{Mujica-book} that $\n{( f_{kl}(x_{ij}) )}_{M_{nm}(W)} \le \n{(x_{ij})}_{M_n(V)}$ (note here we are using that $(f_{kl}) \in B_{M_m(\Hcb(B_V))}$). 
Therefore, we conclude that
\[
\n{(\delta_Vx_{ij})}_{M_n(\Gcb(B_V))} \le \n{(x_{ij})}_{M_n(V)}.
\]
The reverse inequality is clear from the fact that
\[
\n{(x_{ij})}_{M_n(V)} = \sup \{ \n{\big( y_{kl}'(x_{ij})\big)}_{M_{nm}} \;:\; (y_{kl}') \in B_{M_m(V')} \}
\]
and Remark \ref{remark-linear-cb-is-holomorphic-cb}.
\end{proof}

\begin{remark}
In the Banach space setting, the Dixmier-Ng-Kaijser construction can be used to prove linearization results for a number of classes of vector-valued functions (bounded, holomorphic, Lipschitz, holomorphic Lipschitz, etc.).
Similarly, in an upcoming work  \cite{CDD-futuro} we use our Theorem \ref{thm-Kaijser} to prove other linearization results for analogous classes defined on operator spaces such as the completely Lipschitz maps in small scale from \cite{braga2024small} .
\end{remark}

\section{The predual of $\Hcb(B_V)$: de Leeuw's map approach}\label{Sect:De Leeuw}

Recall that $\Hcb(B_V)$ can be completely isometrically isomorphically identified with the image of the map $\Phi_\C : \Hcb(B_V) \to  \ell_\infty\{ M_{n_x} \,:\, x \in \mathbf{B}_V \}$, and note that the latter space is the dual of $\ell_1\{ S_1^{n_x} \,:\, x \in \mathbf{B}_V \}$.

\begin{lemma}\label{prop-de-Leeuws-map}
For any operator space $V$,
the image of $\Phi_\C$ is a weak$^*$ closed subspace of  $\ell_\infty\{ M_{n_x} \,:\, x \in \mathbf{B}_V \}$.
\end{lemma}

\begin{proof}
By the Krein--Smulian theorem, it suffices to check that the unit ball of the image of $\Phi_\C$ is weak$^*$ closed.
Let $(f^{[\alpha]})_{\alpha\in A}$ be a net in the unit ball of $\Hcb(B_V)$ so that $(\Phi_\C(f^{[\alpha]}))_{\alpha\in A}$ converges in the weak$^*$ topology of $\ell_\infty\{ M_{n_x} \,:\, x \in \mathbf{B}_V \}$ to some limit $y$.
In particular, the net $(f^{[\alpha]})_{\alpha\in A}$ converges pointwise to a function $f : B_V \to \C$. Since this net is bounded in $\Hcb(B_V)$, it is also bounded in $\mathcal{H}^\infty(B_V)$ and therefore $f$ is a bounded holomorphic function (because the closed unit ball of $\mathcal{H}^\infty(B_V)$ is compact for the topology of pointwise convergence).
Now, fix $x \in \mathbf{B}_V$. Since $\big( f^{[\alpha]}_{n_x}(x)\big)_{\alpha \in A}$ is a net in $\overline{B}_{M_{n_x}}$ converging pointwise to $f_{n_x}(x)$, we conclude that $f_{n_x}(x) \in \overline{B}_{M_{n_x}}$, so $f$ belongs to the unit ball of $\Hcb(B_V)$. Clearly $y = \Phi_\C(f)$, yielding the desired result.
\end{proof}

The above lemma gives an alternative way to find an abstract predual for $\Hcb(B_V)$, as shown below.
Note that analogous results hold in the classical setting for $\mathcal{H}^\infty(B_V)$ \cite[Sec. 2]{Galindo-Gamelin-Lindstrom}.

\begin{corollary}\label{corollary-predual}
The space $\Hcb(B_V)$ is a dual operator space.
Moreover, on bounded subsets of $\Hcb(B_V)$ the weak$^*$ topology coincides with the topology of pointwise convergence.    
\end{corollary}

\begin{proof}
By the previous lemma and \cite[Prop. 4.2.2]{Effros-Ruan-book}, $\Hcb(B_V)$ is a dual operator space. Moreover, the weak$^*$ topology on  $\Hcb(B_V)$ is the restriction of the weak$^*$ topology from $\ell_\infty\{ M_{n_x} \,:\, x \in \mathbf{B}_V \}$. In particular, convergence in the weak$^*$ topology of  $\Hcb(B_V)$ implies pointwise convergence.
But on the unit ball of $\Hcb(B_V)$ the weak$^*$ topology is compact and the topology of pointwise convergence is Hausdorff, so by standard basic topology the topologies must coincide.
\end{proof}

Note that the map $\Phi_\C$ is playing a role analogous to de Leeuw's map in the context of Lipschitz free spaces (see \cite[Sec. 2.1]{Weaver-Lipschitz-algebras}), hence the name of this section.

We will now show that the predual appearing in the proof above coincides with the one we had previously found with the Dixmier-Ng-Kaijser machinery.
Let $H_{\perp}$ be the preannihilator of $\Phi_\C(\Hcb(B_V))$ in $\ell_1\{ S_1^{n_x} \,:\, x \in \mathbf{B}_V \}$, that is,
\[
H_{\perp} = \big\{ y \in \ell_1\{ S_1^{n_x} \,:\, x \in \mathbf{B}_V \} \;:\; \pair{\Phi_C(f)}{y} = 0 \text{ for all } f \in \Hcb(B_V) \big\}.
\]
Then the predual $\mathscr{G}$ appearing in the proof of Corollary \ref{corollary-predual} is given as
\[
\mathscr{G} = \ell_1\{ S_1^{n_x} \,:\, x \in \mathbf{B}_V \}/H_{\perp}.
\]
Let us denote by $Q$ the associated complete $1$-quotient 
\[
Q :  \ell_1\{ S_1^{n_x} \,:\, x \in \mathbf{B}_V \} \to \mathscr{G}.
\]
Since $\mathscr{G}' = \Hcb(B_V)$, we can completely isometrically isomorphically identify $\mathscr{G}$ with a subspace of $\Hcb(B_V)'$. We find this subspace next.

\begin{proposition}
For any operator space $V$, the space $\mathscr{G}$ can be completely isometrically isomorphically identified with $\mathcal{G}^\infty_{\cb}(B_V)=\overline{\spa}\{ \delta_V(x) \}_{x \in B_V} \subseteq \Hcb(B_V)'$.
\end{proposition}

\begin{proof}
First we will show that for each $x_0\in B_V$, the evaluation map $\delta_V(x_0)$ belongs to the subspace of $\Hcb(B_V)'$ which is identified with $\mathscr{G}$.
To that end, let $e_{x_0} \in \ell_1\{ S_1^{n_x} \,:\, x \in \mathbf{B}_V \}$ be the element that has a 1 in the $x_0$ coordinate (note that the space $S_1^{n_{x_0}}$ is one-dimensional) and $0$'s everywhere else. Then for every $f \in \Hcb(B_V)$,
\[
\pair{Qe_{x_0}}{f} = \pair{e_{x_0}}{\Phi_\C(f)} = f(x_0),
\]
that is, $Qe_{x_0}$ corresponds to $\delta_V(x_0)$, and therefore $\spa\{ \delta_V(x) \}_{x \in B_V}$ is contained in the subspace of $\Hcb(B_V)'$ which is identified with $\mathscr{G}$.
Next, observe that the finitely supported vectors are dense in $\ell_1\{ S_1^{n_x} \,:\, x \in \mathbf{B}_V \}$ and therefore their images under $Q$ are dense in $\mathscr{G}$.
Therefore, it suffices to check that for an $x \in \mathbf{B}_V$ and a matrix $ a = (a_{ij})_{i,j=1}^{n_{x}} \in S_1^{n_{x}}$, the element $y \in \ell_1\{ S_1^{n_x} \,:\, x \in \mathbf{B}_V \} $ having $a$ in the $x$-coordinate and $0$'s everywhere else gets mapped by $Q$ inside $\mathscr{G}$. Now, for every $f \in \Hcb(B_V)$ we have
\[
\pair{Qy}{f} = \pair{y}{\Phi_C(f)} = \sum_{i,j=1}^{n_{x}} a_{ij} f(x_{ij})
\]
that is, the element of $\Hcb(B_V)'$ corresponding to $Qy$ is $ \sum_{i,j=1}^{n_{x}} a_{ij} \delta_V(x_{ij}) \in \mathcal{G}^\infty_{\cb}(B_V)$ 
The desired conclusion follows.
\end{proof}

A typical issue in classical linearization procedures is to describe the unit ball of the predual of the space of functions as the closed absolutely convex hull of the elementary elements. In our setting this kind of result holds for the matrix unit ball. We present two alternative descriptions.

\begin{proposition}\label{proposition-ball-is-convex-hull-of-image-of-delta}
Let $V$ be an operator space.
Let $\mathbf{A} = (A_n)_{n=1}^\infty$ be the matrix set over $\Gcb(B_V)$ given by, for each $n\in\N$, $A_n = \big\{ ((\delta_V)_n(x)) \;:\; x \in B_{M_n(V)}\big\}$.
Then the closed matrix unit ball of $\Gcb(B_V)$ is the closure of the absolutely matrix convex hull of $\mathbf{A}$.

Moreover, if $\widetilde{\mathbf{A}} = (\widetilde{A}_n)_{n=1}^\infty$ is the matrix set over $\Gcb(B_V)$ given by $\widetilde{A}_n = \big\{ (\delta_V)_n(x)/\n{x} \;:\; x \in B_{M_n(V)} \setminus \{0\}\big\}$ for each $n$, then we also have that the closed matrix unit ball of $\Gcb(B_V)$ is the closure of the absolutely matrix convex hull of $\widetilde{\mathbf{A}}$.
\end{proposition}

\begin{proof}
From the definition of the operator space structure on $\Hcb(B_V)$, note that for any $f \in M_m(\Hcb(B_V))$,
\begin{equation}\label{eqn-completely-norming}
\n{f}_{M_m(\Hcb(B_V))} = \sup \big\{ \n{\mpair{f}{a}}_{M_{nm}} \;:\; a \in A_n, n\in\mathbb N \big\}.
\end{equation}
Since $\mathbf{A} \subseteq \mathbf{B}_{\Gcb(B_V)}$, the closure of the absolutely matrix convex hull of $\mathbf{A}$ is contained in $\overline{\mathbf{B}}_{\Gcb(B_V)}$. If they were different, the Hahn--Banach theorem for matrix convexity (Theorem \ref{thm-Hahn-Banach-matrix-convexity}) would yield a contradiction with \eqref{eqn-completely-norming}.

For the second part recall that the elements of $\widetilde{A}_n$ have norm one thanks to Proposition \ref{prop-delta-is-cb-holomorphic} and note that from Corollary \ref{corollary-Schwarz} $(b)$, for any $f \in M_m(\Hcb(B_V))$ and any $x \in B_{M_n(V)} \setminus \{0\}$ we have 
$\n{\mpair{f}{  (\delta_V)_n(x)/\n{x}}}_{M_{nm}} \le \n{f}_{M_m(\Hcb(B_V))} \cdot 1$, from where
\[
\n{f}_{M_m(\Hcb(B_V))} = \sup \big\{ \n{\mpair{f}{a}}_{M_{nm}} \;:\; a \in \widetilde{A}_n, n\in\mathbb N \big\}.
\]
Now, the rest of the proof is analogous to that of the first part.
\end{proof}

Note that from Proposition \ref{proposition-ball-is-convex-hull-of-image-of-delta}, since the $n$-th level of $\amconv(\mathbf{A})$ (resp. $\amconv(\widetilde{\mathbf{A}})$) is dense in $\overline{B}_{M_n(\Gcb(B_V))}$, a standard argument shows that every $u \in B_{M_n(\Gcb(B_V))}$ can be written as an absolutely convergent series $u = \sum_{j=1}^\infty u_j$ where $u_j \in \amconv(\mathbf{A})$ (resp. $\amconv(\widetilde{\mathbf{A}})$) and $\sum_{j=1}^\infty \n{u_j}_{M_n(\Gcb(B_V))} < 1$.
The next result is a slight refinement of this, compare with the description of the absolutely matrix convex hull given in Lemma \ref{lemma-characterization-abs-mat-convex-hull}.

\begin{proposition}\label{prop-description-oss-Gcb}
Let $V$ be an operator space.
Every $u \in M_n(\Gcb(B_V))$ admits a representation as an absolutely convergent series of the form
\[
u = \sum_{j=1}^\infty \left( \sum_{i=1}^{m_j} \alpha^j_i c^j_i(\delta_V)_{k^j_i}(x^j_i) \beta^j_i \right)
\]
where for each $j$ 
we have $x^j_i \in B_{M_{k^j_i}(V)}$, $\alpha^j_i \in M_{n,k^j_i}$, $\beta^j_i \in M_{k^j_i,n}$, $c^j_i \in \C$.
Moreover, 
\[
\n{u}_{M_n(\Gcb(B_V))} = \inf \bigg\{ \sum_{j=1}^\infty \n{\sum_{i=1}^{m_j} \alpha^j_i(\alpha^j_i)^*}^{1/2} \n{\sum_{i=1}^{m_j} (\beta^j_i)^*\beta^j_i}^{1/2} \max_{1 \le i \le m_j} |c^j_i|\n{x^j_i} \bigg\}
\]
where the infimum is taken over all such representations.
\end{proposition}

\begin{proof}
Denote $G = \spa\{\delta_V(x)\}_{x\in B_V} \subset \Gcb(B_V)$. For any $u \in M_n(G)$ define
\[
\tn{u}_n = \inf \bigg\{ \n{\sum_{i=1}^m \alpha_i\alpha_i^*}^{1/2} \n{\sum_{i=1}^m \beta_i^*\beta_i}^{1/2} \max_{1 \le i \le m} |c_i|\n{x_i} \bigg\}
\]
where the infimum is taken over all representations of the form 
$u = \sum_{i=1}^m \alpha_i c_i(\delta_V)_{k_i}(x_i) \beta_i$ with $x_i \in B_{M_{k_i}(V)}$, $\alpha_i \in M_{n,k_i}$, $\beta_i \in M_{k_i,n}$, $c_i \in \C$.
Note that the representations under consideration are equivalent to having
\[
u = \begin{bmatrix}
\alpha_1 & \alpha_2 &\cdots &\alpha_m\\
\end{bmatrix}
\begin{bmatrix}
c_1 (\delta_V)_{k_1}(x_1) & 0 &  \cdots & 0\\
0 & c_2 (\delta_V)_{k_2}(x_2) &  \cdots & 0\\
\vdots &  & \ddots  & \vdots\\
0 & 0 &  \cdots & c_m (\delta_V)_{k_m}(x_m)\\
\end{bmatrix}
\begin{bmatrix}
\beta_1 \\ \beta_2 \\ \vdots \\ \beta_m\\
\end{bmatrix}
\]
from where it is clear that $\n{u}_{M_n(\Gcb(B_V))} \le \tn{u}_n$, using Ruan's axioms for $\Gcb(B_V)$ and the fact that by Proposition \ref{prop-delta-is-cb-holomorphic},
 for any $x \in B_{M_n(V)}$ we have $\n{ (\delta_V)_n(x) }_{M_n(\Gcb(B_V))} = \n{x}_{M_n(V)}$.
Additionally, note that using a trivial representation we have $\tn{ (\delta_V)_n(x) }_n \le \n{x}_{M_n(V)}$, and therefore $\tn{ (\delta_V)_n(x) }_n = \n{x}_{M_n(V)}$.

Let us now show that the sequence $(\tn{\cdot})_{n=1}^\infty$ satisfies the criteria in \cite[Prop. 2.3.6]{Effros-Ruan-book}, which will imply  that $(\tn{\cdot}_n)_{n=1}^\infty$ defines a sequence of seminorms on $G$ that satisfies Ruan's axioms.
Let $u \in M_n(G)$, $\alpha \in M_{\ell,n}$, $\beta \in M_{n,\ell}$. A representation $u = \sum_{i=1}^m \alpha_i c_i(\delta_V)_{k_i}(x_i) \beta_i$ as in the definition of $\tn{\cdot}_n$ yields   $\alpha u \beta=\sum_{i=1}^m \alpha\alpha_i c_i(\delta_V)_{k_i}(x_i) \beta_i\beta$. Since
\[
\n{\sum_{i=1}^m \alpha\alpha_i\alpha_i^*\alpha^*}^{1/2} \le \n{\sum_{i=1}^m \alpha_i\alpha_i^*}^{1/2} \n{\alpha}, \qquad \n{\sum_{i=1}^m \beta^*\beta_i^*\beta_i \beta}^{1/2} \le \n{\sum_{i=1}^m \beta_i^*\beta_i}^{1/2}\n{\beta},
\]
it is clear now that $\tn{ \alpha u \beta}_\ell \le \n{\alpha} \tn{u}_n \n{\beta}$.
Now take $u^1\in M_{n_1}(G)$, $u^2\in M_{n_2}(G)$  and let $\varepsilon>0$ be given.
For each $j=1,2$, choose representations $u^j = \sum_{i=1}^{m_j} \alpha^j_i c^j_i(\delta_V)_{k^j_i}(x^j_i) \beta^j_i$ as in the definition of $\tn{u^j}_{n_j}$ with $\max_{1 \le i \le m_j} |c^j_i|\n{x^j_i} =1$ and $\n{\sum_{i=1}^{m_j} \alpha^j_i(\alpha^j_i)^*}^{1/2}=\n{\sum_{i=1}^{m_j} (\beta^j_i)^*\beta^j_i}^{1/2}\le (\tn{u^j}_{n_j} + \varepsilon)^{1/2}$ and  (this can be achieved by rescaling).
Note that $u_1 \oplus u_2$ can be written as
\[
u_1 \oplus u_2=\sum_{i=1}^{m_1} \begin{bmatrix} \alpha^1_i\\ 0 \end{bmatrix} c^1_i(\delta_V)_{k^1_i + n_2}(x^1_i \oplus 0) \begin{bmatrix} \beta^1_i &0 \end{bmatrix} + \sum_{i=1}^{m_2} \begin{bmatrix}0 \\ \alpha^2_i \end{bmatrix} c^2_i(\delta_V)_{n_1+ k^2_i }(0 \oplus x^2_i ) \begin{bmatrix}0 & \beta^2_i \end{bmatrix}.
\]
Now, since
\[
\sum_{i=1}^{m_1} \begin{bmatrix} \alpha^1_i\\ 0 \end{bmatrix} \begin{bmatrix} \alpha^1_i\\ 0 \end{bmatrix}^* + \sum_{i=1}^{m_2} \begin{bmatrix}0 \\ \alpha^2_i \end{bmatrix}\begin{bmatrix}0 \\ \alpha^2_i \end{bmatrix}^* = \left( \sum_{i=1}^{m_1}  \alpha^1_i(\alpha^1_i)^* \right) \oplus \left( \sum_{i=1}^{m_2}  \alpha^2_i(\alpha^2_i)^* \right)
\]
it follows that
\[
\n{ \sum_{i=1}^{m_1} \begin{bmatrix} \alpha^1_i\\ 0 \end{bmatrix} \begin{bmatrix} \alpha^1_i\\ 0 \end{bmatrix}^* + \sum_{i=1}^{m_2} \begin{bmatrix}0 \\ \alpha^2_i \end{bmatrix}\begin{bmatrix}0 \\ \alpha^2_i \end{bmatrix}^* }^{1/2} \le (\max\big\{\tn{u^1}_{n_1}, \tn{u^2}_{n_2}\big\} + \varepsilon)^{1/2}
\]
and similarly 
\[
\n{ \sum_{i=1}^{m_1} \begin{bmatrix} \beta^1_i &0 \end{bmatrix}^* \begin{bmatrix} \beta^1_i &0 \end{bmatrix} + \sum_{i=1}^{m_2} \begin{bmatrix}0 & \beta^2_i \end{bmatrix}^* \begin{bmatrix}0 & \beta^2_i \end{bmatrix} }^{1/2} \le (\max\big\{\tn{u^1}_{n_1}, \tn{u^2}_{n_2}\big\} + \varepsilon)^{1/2}
\]
from where we conclude that
$\tn{ u^1 \oplus u^2 }_{n_1+n_2} \le \max\big\{\tn{u^1}_{n_1}, \tn{u^2}_{n_2}\big\}$.

We next note that the seminorms $\tn{\cdot}_n$ are in fact norms, because $\tn{\cdot}_n$ dominates the norm on $M_n(\Gcb(B_V))$.
Therefore, if we denote by $\widetilde{G}$ the completion of $G$ under $\tn{\cdot}_1$, then the formal identity $\iota : G \to \Gcb(B_V)$ extends to a complete contraction $\iota : \widetilde{G} \to \Gcb(B_V)$
and therefore $\iota(\mathbf{B}_{\widetilde{G}}) \subseteq \mathbf{B}_{\Gcb(B_V)}$.
Recall that for every $n\in \N$ and every $x \in B_{M_n(V)}$ we have that $\tn{ (\delta_V)_n(x) }_n = \n{x}_{M_{n}(V)} < 1$, so $\mathbf{A} \subset \iota(\mathbf{B}_{\widetilde{G}})$ where $\mathbf{A}$ is defined as in Proposition \ref{proposition-ball-is-convex-hull-of-image-of-delta}.
Since $\mathbf{B}_{\widetilde{G}}$ is absolutely matrix convex and thus so is $\iota(\mathbf{B}_{\widetilde{G}})$, it follows from Proposition \ref{proposition-ball-is-convex-hull-of-image-of-delta} that
$
\overline{ \iota(\mathbf{B}_{\widetilde{G}}) } = \overline{\mathbf{B}}_{\Gcb(B_V)},
$
and  then, by standard arguments
(see e.g. \cite[Proof of Prop. 4.28]{Bowers-Kalton})
 it yields that
$
\iota( \overline{\mathbf{B}}_{\widetilde{G}})  \supseteq \mathbf{B}_{\Gcb(B_V)}.
$
Multiplying by $\lambda\in(0,1)$ and taking the union, we get
$
\iota( \mathbf{B}_{\widetilde{G}})  \supseteq \mathbf{B}_{\Gcb(B_V)},
$
from where we obtain that 
$
\iota( \mathbf{B}_{\widetilde{G}})  = \mathbf{B}_{\Gcb(B_V)}.
$
So, $\iota$ is not just a contraction but in fact a complete $1$-quotient, and now the desired conclusion follows by the standard arguments mentioned in the paragraph before the statement of this proposition.
\end{proof}

\section{The linearization property}

We are now ready to prove that the space $\Gcb(B_V)$ satisfies the desired linearization property.

\begin{theorem}\label{thm-linearization}
Let $V$ be an operator space.
For every operator space $W$ and every completely bounded holomorphic function $f:B_V\to W$, there exists a unique operator $\widehat{f}:\Gcb(B_V)\to W$ such that $\widehat{f}\circ \delta_V=f$ and $\bn{\widehat{f}}_{\cb}=\n{f}_{ \Hcb(B_V,W) }$. That is, $\widehat{f}$ makes the following diagram commutative:
\[
\xymatrix{\Gcb(B_V)\ar@{-->}[rd]^{\widehat{f}}\\
     B_V\ar[r]^{f}\ar^{\delta_V}[u]& W }
\]
This property completely isometrically isomorphically characterizes $\Gcb(B_V)$.
Moreover, the assignment $f \mapsto \widehat{f}$ is a completely isometric isomorphism between $\Hcb(B_V,W)$ and $\CB(\Gcb(B_V),W)$.
\end{theorem}

\begin{proof}
Arguing exactly as in \cite[Prop. 2.5(c)]{aron2023linearization}, we note that the set $\{ \delta_V(x) \}_{x \in B_V\setminus\{0\}}$ is linearly independent.
Therefore, we can define $ \widehat{f} : \spa \{ \delta_V(x) \}_{x \in B_V} \to W$ by $\widehat{f}(\delta_V(x)) = f(x)$ and extending linearly.
From Proposition \ref{proposition-ball-is-convex-hull-of-image-of-delta} and Lemma \ref{lemma-norm-of-abs-mat-conv-hull}, in order for $\widehat{f}$ to extend to a completely bounded linear map from $\Gcb(B_V)$ to $W$ it suffices that the quantity
$\sup\big\{ \bn{ \big( \widehat{f}\delta_V(x_{ij})  \big) }_{M_n(W)} \;:\; (x_{ij}) \in B_{M_n(V)} \big\}$ is finite, and this quantity will in fact be the $\cb$-norm of the extension. But from our definition of $\widehat{f}$ this is
$\sup\big\{ \bn{ \big( f(x_{ij})  \big) }_{M_n(W)} \;:\; (x_{ij}) \in B_{M_n(V)} \big\}$, which is precisely $\n{f}_{\Hcb}$.
\end{proof}

\begin{remark}\label{remark-linealization-preserves-properties}
Note that Theorem \ref{thm-linearization} easily implies an operator space version of \cite[Prop. 3.1(a)]{Mujica-linearization}: a function $f \in \Hcb(B_V,W)$ has image contained in a finite-dimensional subspace of $W$ if and only if $\widehat{f}$ is a finite-rank linear map.
This follows from looking at the diagram above using that $\Gcb(B_V)=\overline{\spa}\{\delta_V(x)\}_{x\in B_V}$.
In particular this applies when $f$ is a finite type polynomial, since as we pointed out in Section \ref{sec:prelim} in this case the image of $f$ is contained in a finite-dimensional subspace.

More generally, Proposition \ref{proposition-ball-is-convex-hull-of-image-of-delta} allows one to get various operator space versions of \cite[Prop. 3.4(b)]{Mujica-linearization} showing that the linearization process in Theorem \ref{thm-linearization} preserves various operator space notions of compactness. Since we will not be needing these results, we only sketch the argument and do not introduce all the relevant concepts (see \cite{ChaDiGa-Operator-p-compact, Webster} for the missing definitions).
We say that a function $f : B_V \to W$ is operator $p$-compact (resp. completely compact) if the matrix set over $W$ given by $\big( f_n(B_{M_n(V)}) \big)_{n=1}^\infty$ is operator $p$-compact (resp. completely compact).
Then, $f \in \Hcb(B_V,W)$ is operator $p$-compact (resp. completely compact) if and only if $\widehat{f}$ is an operator $p$-compact (resp. completely compact) linear map.
The key is to observe that for each $n\in\N$ we have $f_n(B_{M_n(V)}) = \widehat{f}_nA_n$ where $A_n$ is the set in Proposition \ref{proposition-ball-is-convex-hull-of-image-of-delta}, and that the property of being operator $p$-compact (resp. completely compact) is preserved when taking absolutely matrix convex hulls and closures. 
\end{remark}

As in the classical setting, we also have the following standard alternative linearization result:

\begin{theorem}
Let $V$ and $W$ be operator spaces.
If $f : B_V \to B_W$ is a completely bounded holomorphic function, there exists a unique completely bounded linear map $\widetilde{f} : \Gcb(B_V) \to \Gcb(B_W)$ such that $\widetilde{f}\delta_V = \delta_Wf$, that is, the following diagram commutes:
\[
\xymatrix{
\Gcb(B_V)\ar[r]^{\widetilde{f}} &\Gcb(B_W)\\
     B_V\ar[r]^{f}\ar[u]^{\delta_V}& B_W \ar[u]^{\delta_W}
     }
\]
Moreover, $\bn{\widetilde{f}}_{\cb} = \n{f}_{\Hcb}$.
\end{theorem}

\begin{proof}
Note that the uniqueness part of the statement is clear, since $\widetilde{f}$ is uniquely determined by its values on the image of $\delta_V$ which has dense span in $\Gcb(B_V)$.
Note that since both $f$ and $\delta_W$ are completely bounded holomorphic functions, so is their composition $\delta_W f : B_V \to \Gcb(B_W)$.
Therefore, by the linearization property there exists $\widetilde{f} \in \CB(\Gcb(B_V), \Gcb(B_W))$ such that $\widetilde{f} \delta_V = \delta_W f$ and moreover $\bn{\widetilde{f}}_{\cb} = \n{\delta_W f }_{\Hcb}$.
Since $\delta_W$ is completely isometric (see Proposition  \ref{prop-delta-is-cb-holomorphic}), clearly $\n{\delta_W f }_{\Hcb} = \n{f}_{\Hcb}$.
\end{proof}

Now we present an operator space version of \cite[Prop. 2.3]{Mujica-linearization}.

\begin{proposition}\label{Prop:complemented subspace}
Any operator space
$V$ is completely isometrically isomorphic to a completely contractively complemented subspace of $\mathcal{G}^\infty_{\cb}(B_V)$.    
\end{proposition}

\begin{proof}
Consider the identity function $Id_V : V \to V$, whose restriction to $B_V$ is a completely bounded holomorphic function since $Id_V$ is a completely bounded linear function. We then get the associated diagram given by the universal property of $\Gcb(B_V)$:
\[
\xymatrix{\Gcb(B_V)\ar[rd]^{\widehat{Id_V}}\\
     B_V\ar[r]^{Id_V}\ar^{\delta_V}[u]& V }
\]
Since all the functions involved are holomorphic, and in fact $Id_V$ and $\widehat{Id_V}$ are linear, taking derivative at $0$ we get
\[
\xymatrix{\Gcb(B_V)\ar[rd]^{\widehat{Id_V}}\\
V\ar[r]^{Id_V}\ar^{d\delta_V(0)}[u]& V }
\]
Note that 
$\bn{\widehat{Id_V}}_{\cb} = \n{Id_V}_{\Hcb} \le 1$ and by Lemma \ref{lemma-Taylor-polynomials-at-0-are-Schur} $\n{d\delta_V(0)}_{\cb} \le \n{\delta_V}_{\Hcb} \le 1$, so $V$ is completely isometric to the subspace $d\delta_V(0)V$ of $\mathcal{G}^\infty_{\cb}(B_V)$.
Moreover, note that $d\delta_V(0)\widehat{Id_V}$ is a completely contractive projection onto $d\delta_V(0)V$.
\end{proof}

We now observe that our operator space $\Gcb(B_V)$ in general cannot be the trivial suggestion $\MAX\big(\mathcal{G}^\infty(B_V)\big)$ mentioned in the introduction. 

\begin{corollary}\label{cor-Gcb-not-maximal}
If $V$ is an operator space which is not maximal, then $\Gcb(B_V)$ cannot be maximal (in particular, it cannot be completely isometrically isomorphic to  $\MAX\big(\mathcal{G}^\infty(B_V)\big)$) and therefore $\Hcb(B_V)$ cannot be minimal.   
\end{corollary}

\begin{proof}
Suppose that $\Gcb(B_V)$ is a maximal operator space. By Proposition  \ref{Prop:complemented subspace} we get that $V$ is also maximal, a contradiction.
\end{proof}

\section{Transference of approximation properties between $V$ and $\Gcb(B_V)$}\label{section-OAP}

Among the main results of \cite{Mujica-linearization} is that certain classical approximation properties transfer between $X$ and $\mathcal{G}^\infty(B_X)$. In this section we prove an operator space version of those results, using noncommutative adaptations of the simplification of Mujica's arguments as presented in \cite{aron2023linearization}.

Recall \cite[Sec. 11.2]{Effros-Ruan-book} that an operator space $V$ is said to have the \emph{Operator Approximation Property} if the identity of $V$ is the limit of a net of finite rank maps in $\CB(V,V)$ with respect to the stable point norm topology. For the purposes of the present paper one can take the following result of Webster \cite[Prop. 4.3.1]{Webster} as the definition of the aforementioned topology, which we will denote by $\tau_0$.

\begin{proposition}\label{prop-stable-point-norm}
Let $V$ and $W$ be operator spaces, and $T_\alpha, T \in \CB(V,W)$. The following are equivalent:
\begin{enumerate}[(a)]
    \item $T_\alpha \to T$ in the stable point norm topology on $\CB(V,W)$.
    \item $(T_\alpha)_\infty \to T_\infty$ uniformly on the compact subsets of $\K(V)$.
\end{enumerate}
\end{proposition}

Since we are aiming to adapt Mujica's arguments in the classical case, which are based on considering various vector topologies on the spaces involved, we need to define corresponding topologies in our setting. Some of those correspondences are already well established: for example, the topology of uniform convergence over compact subsets on $\mathcal{B}(V,V)$ (which is used to define the classical AP) corresponds to the stable point norm topology on $\CB(V,V)$ (which is used to define the OAP).
Now, Mujica also considers the topology of uniform convergence over compact subsets on $\mathcal{H}^\infty(B_V,W)$.
In analogy with Proposition \ref{prop-stable-point-norm}, we introduce the following definition.

\begin{definition}\label{defn-tau_0}
Let $V$ and $W$ be  operator spaces. Given $f_\alpha, f \in \Hcb(B_V,W)$, we say that $f_\alpha$ converges to $f$ in the topology $\tau_0$ if $(f_\alpha)_\infty \to f_\infty$ uniformly on the compact subsets of $B_{\K(V)}$.
\end{definition}

We emphasize that the above definition makes sense thanks to Corollary \ref{corollary-Schwarz}, which says that the infinite amplification of a function $f \in \Hcb(B_V,W)$ defines a function from $B_{\K(V)}$ to $\K(W)$.
Note also that while at first sight it could look like an abuse of notation to use the symbol $\tau_0$ to denote topologies on the very different spaces $\CB(V,W)$ and $\Hcb(B_V,W)$, from Remark \ref{remark-linear-cb-is-holomorphic-cb} we have a natural completely isometric embedding $\CB(V,W) \hookrightarrow \Hcb(B_V,W)$ and in fact the original topology $\tau_0$ on $\CB(V,W)$ is simply the restriction of the new one defined 
on $\Hcb(B_V,W)$.

The following is a version of Grothendieck's celebrated characterization of compactness in a normed space, and more precisely of the refinement stated in \cite[Lemma 3.4]{aron2023linearization}.

\begin{proposition}\label{prop-Grothendieck-abs-mat-conv}
Suppose that for an operator space $G$ and a matrix set $\mathbf{K}$ over $G$, $\overline{\mathbf{B}}_G = \overline{\amconv}(\mathbf{K})$.
Then, for a given compact set $Y \subset \K(G)$, there exists a sequence  $(a_j x_j)$ with $a_j>0$, $a_j \to 0$ and $x_j \in \mathbf{K}$, such that $Y \subseteq \overline{\amconv\{ a_j x_j \} }^{\K(G)}$.
\end{proposition}

\begin{proof}
Fix a sequence $(\varepsilon_n)$ of positive numbers converging to 0.
Let us assume, without loss of generality, that $Y_1:=Y \subseteq \overline{B}_{\K(G)}$.
Note that $2Y_1 \subset 2\overline{B}_{\K(G)}$ is compact, therefore there exists a finite set $F_1 \subseteq 2 \overline{B}_{\K(G)}$ such that $2Y \subseteq \bigcup_{z \in F_1} \overline{B}(z,\varepsilon_1/2)$.
Now, for each $z_i \in F_1$ by truncating and using Lemma \ref{lemma-characterization-abs-mat-convex-hull}, there exists a sum of the form $z'_i=\sum_{j=1}^{m_i} \alpha^1_{i,j} 2x^1_{i,j} \beta^1_{i,j}$ with $\sum_{j=1}^{m_i} \alpha^1_{i,j}(\alpha^1_{i,j})^* \le \mathbbm{1}_{n_i}, \sum_{j=1}^{m_i} (\beta^1_{i,j})^*\beta^1_{i,j} \le \mathbbm{1}_{n_i}$, $x^1_{i,j} \in \mathbf{K}$ and $\n{z_i-z'_i} < \varepsilon_1/2$.
If we set $a^1_{i,j} = 2$ for $1 \le i \le |F_1|$ and $1 \le j \le m_i$, what we have is that there exists a finite set $G_1 \subseteq \amconv\{ a^1_{i,j} x^1_{i,j} \}$ such that
$2Y_1 \subseteq \bigcup_{z \in G_1} \overline{B}(z, \varepsilon_1)$.

Now we consider
$Y_2 = \bigcup_{z \in G_1}\big( 2Y_1 \cap \overline{B}(z,\varepsilon_1) - z \big)$.
Notice that $Y_2$ is a compact subset of $\varepsilon_1 \overline{B}_{\K(G)}$.
As before, there exists a finite set $F_2 \subset 2Y_2 \subset 2\varepsilon_1 \overline{B}_{\K(G)}$ 
such that $2Y_2 \subseteq \bigcup_{z \in F_2} \overline{B}(z,\varepsilon_2/2)$.
By the same arguments as before, setting $a^2_{i,j} = 2\varepsilon_1$ for $1 \le i \le |F_2|$, %and $1 \le j \le m_i$, 
there exist $x^2_{i,j} \in \mathbf{K}$ and a finite set $G_2 \subseteq \amconv\{ a^2_{i,j} x^2_{i,j} \}$ such that
$2Y_2 \subseteq \bigcup_{z \in G_2} \overline{B}(z,\varepsilon_2)$.

We continue the process inductively, at each step defining
$Y_{n+1} = \bigcup_{z \in G_n}\big( 2Y_n \cap \overline{B}(z,\varepsilon_n) - z \big) \subseteq \varepsilon_n \overline{B}_{\K(G)}$,
setting $a^{n+1}_{i,j} = 2\varepsilon_n$ and finding 
 $x^{n+1}_{i,j} \in \mathbf{K}$ and a finite set $G_{n+1} \subseteq \amconv\{ a^{n+1}_{i,j} x^{n+1}_{i,j} \}$ such that
$2Y_{n+1} \subseteq \bigcup_{z \in G_{n+1}} \overline{B}(z, \varepsilon_{n+1})$.

We claim that the sequence $(a_jx_j)$ obtained as the concatenation of the blocks $\{ a^n_{i,j} x^n_{i,j} \}$ for $n=1,2,3,\dotsc$ satisfies the desired conclusion. Since $a^n_{i,j} = 2 \varepsilon_n$, we do have $a_j \to 0$.
Now consider  a fixed $y \in Y=Y_1$.
Since $2y \in 2Y_1$, there exists $z_1 \in G_1$ such that $2y-z_1 \in Y_2$.
Then $4y-2z_1 \in 2Y_2$, so there exists $z_2 \in G_2$ such that $4y-2z_1-z_2 \in Y_3$.
Inductively, once we have chosen $z_n \in G_n$ we then find $z_{n+1} \in G_{n+1}$ such that 
$2^{n+1}y -2^{n}z_1 - \cdots -2z_{n} - z_{n+1} \in Y_{n+2}$.

Since all of the $Y_n$ are uniformly bounded by $\sup_j\varepsilon_j$, it follows that the series $\sum_{n=1}^\infty \frac{z_n}{2^n}$ converges to $y$.
The partial sums of this series are finite convex combinations of elements of $\amconv(\mathbf{K})$, which implies that these partial sums are themselves in $\amconv(\mathbf{K})$, and therefore $y \in \overline{\amconv\{ a_j x_j \} }^{\K(G)}$.
\end{proof}

By the linearization property, we can identify  $\Hcb(B_V,W)$ with $\CB(\Gcb(B_V),W)$.
In our next result we describe the topology on the former space that corresponds to $\tau_0$ on the latter one.
Compare to \cite[Thm. 4.8]{Mujica-linearization} (see also \cite[Thm. 3.6]{aron2023linearization}).

\begin{theorem}\label{thm-isomorphismtau-gamma-tau0}
For operator spaces $V$ and $W$, we let $\tau_\gamma$ denote the locally convex vector topology on $\Hcb(B_V,W)$ generated by all the seminorms of the form
\[
p(f) = \sup_{j} a_j \n{ f_{n_j}(x_j) }_{M_{n_j(W)}}
\]
where $n_j \in \N$, $x_j \in B_{M_{n_j}(V)}$, $a_j>0$ for all $j$ and $a_j \to 0$.
Then 
\[
f \in \big( \Hcb(B_V,W), \tau_\gamma \big) \mapsto \widehat{f} \in \big(\CB(\Gcb(B_V),W), \tau_0 \big)
\]
is a topological isomorphism.
\end{theorem}

\begin{proof}
If $Y \subset \K(\Gcb(B_V))$ is a compact set, by Propositions \ref{proposition-ball-is-convex-hull-of-image-of-delta} and  \ref{prop-Grothendieck-abs-mat-conv} there exist sequences  $(a_j)$ and $(x_j)$ with $a_j>0$, $a_j \to 0$, $x_j \in B_{M_{n_j}(V)}$, such that $Y \subseteq \overline{\amconv\{ a_j (\delta_V)_{n_j}(x_j) \} }^{\K(G)}$.
Then, for $f \in \Hcb(B_V,W)$ we have, using Lemma \ref{lemma-norm-of-abs-mat-conv-hull}
\[
\sup _{y \in Y} \n{ \widehat f_\infty(y) }_{\K(W)} \le \sup _{y \in \amconv\{ a_j (\delta_V)_{n_j}(x_j) \}} \n{ \widehat f_\infty(y) }_{\K(W)}  = \sup _{j} a_j \n{ f_{n_j}(x_j) }_{M_{n_j}(W)}
\]
which shows that the map $f \mapsto \widehat{f}$ is $\tau_\gamma$-$\tau_0$ continuous.

Now consider a seminorm $p$ as in the definition of $\tau_\gamma$.
Note that the associated sequence $\big( a_j (\delta_V)_{n_j}(x_j) \big)$ converges to $0$ in $\K(\Gcb(B_V))$.
Therefore, $Y = \{0\} \cup \big\{ a_j (\delta_V)_{n_j}(x_j) \big\}$ is a compact set in $\K(\Gcb(B_V))$, and thus for $f \in \Hcb(B_V,W)$
\[
p(f) =  \sup _{j} a_j \n{ f_{n_j}(x_j) }_{M_{n_j}(W)} = \sup _{y \in Y} \n{ \widehat{f}_\infty(y) }_{\K(W)} 
\]
which shows the $\tau_0$-$\tau_\gamma$ continuity of the inverse mapping.
\end{proof}

\begin{remark}\label{remark-tau-gamma-with-sequences-bounded-by-1}
Note that in the definition of $\tau_\gamma$ above one can add the extra condition that $a_j \le 1$ for all $j\in\N$, and this still defines the same topology.     
\end{remark}

\begin{remark}\label{rmk:tau_0 topology}
For a bounded net $(T_\alpha)$ in $\CB(V,W)$, pointwise convergence implies convergence in the $\tau_0$ topology.
To see this, suppose such a net converges pointwise to $T \in \CB(V,W)$.
First, by truncating (and using the boundedness of $(T_\alpha)$) one can prove that the amplifications $(T_\alpha)_\infty$ converge pointwise to $T_\infty$ on $\K(V)$.
Since the amplifications $(T_\alpha)_\infty$ are also uniformly bounded, the standard argument taking a fine enough finite subset shows that $(T_\alpha)_\infty$ will converge uniformly to $T_\infty$ on compact subsets of $\K(V)$.

Therefore, for a bounded net $(f_\alpha)$ in $\Hcb(B_V,W)$ pointwise convergence will imply $\tau_\gamma$ convergence.
To see this, suppose that $f$ is the pointwise limit of $(f_\alpha)$. By linearity, it follows that $\widehat{f}_\alpha(u)$ converges to $\widehat{f}(u)$ for each $u \in \spa\{ \delta_V(x) \}_{x \in B_V}$. Since the net $(f_\alpha)$ is bounded so is $(\widehat{f}_\alpha)$, which yields that $\widehat{f}_\alpha(u)$ converges to $\widehat{f}(u)$ for all $u \in \Gcb(B_V)$. By the previous paragraph $(\widehat{f}_\alpha)$ converges to $\widehat{f}$ in $\tau_0$, that is,  $(f_\alpha)$ converges to $f$ in $\tau_\gamma$.
\end{remark}

Let $\mathcal P_{s,0}(V,W)$ the operator subspace of $\mathcal P_{s}(V,W)=\bigcup_m \mathcal P_{s}(^mV,W)$ consisting of polynomials $P$ such that $P(0)=0$. Note that $\mathcal P_{s,0}(V,W)$ is completely isometrically included in $\Hcb(B_V,W)$. We also denote by $\mathcal P_{s,f,0}(V,W)$ the subspace of \textit{finite type} polynomials in $\mathcal P_{s,0}(V,W)$.

Recall that an operator space $V$ is said to have the \emph{Completely Metric Approximation Property (CMAP)} if there exists a net of finite rank complete contractions in $\CB(V,V)$ that converges pointwise to the identity of $V$.
Before proving that this property transfers between $V$ and $\Gcb(B_V)$, we start with a preparatory lemma.

\begin{lemma}\label{lemma:approximation-by-polynomials}
    Let $V$ and $W$ be operator spaces. Then
    \begin{enumerate}[(a)]
        \item For $f\in\overline{B}_{\Hcb(B_V,W)}$ there exists a sequence $(\sigma_mf)\subset 
         \overline{B}_{\mathcal P_{s,0}(V,W)}$ such that $\sigma_mf(x)\to f(x)$ for all $x\in B_V$.
        \item $\overline{B}_{\Hcb(B_V,W)}= \overline{B_{\mathcal P_{s,0}(V,W)}}^{\tau_\gamma}$.
        \item If $V$ has the CMAP then $\overline{B}_{\Hcb(B_V,W)}= \overline{B_{\mathcal P_{s,f,0}(V,W)}}^{\tau_\gamma}$.
    \end{enumerate}
\end{lemma}

 \begin{proof}
     (a) Is just an adaptation of the proof of \cite[Prop. 5.2]{Mujica-linearization}, using Lemma \ref{lemma-Taylor-polynomials-at-0-are-Schur} to control the $\Hcb$-norm of the Taylor polynomials of $f$ at $0$.
    
(b) From (a) we have a bounded sequence in $\mathcal P_{s,0}(V,W) \subset \Hcb(B_V,W)$ converging pointwise to $f$, Remark \ref{rmk:tau_0 topology} now yields that the convergence is also in $\tau_\gamma$.

(c) By (b) it is enough to consider $P\in B_{\mathcal P_{s,0}(V,W)}$ and prove that it can be approximated in the $\tau_0$ topology by a net in $B_{\mathcal P_{s,f,0}(V,W)}$. Since $V$ has the CMAP we know that there is a net $(T_\alpha)\subset \overline B_{\CB(V,V)}$ of finite rank mappings such that $\|T_\alpha (x)-x\|\to 0$ for every $x\in V$. Then, the net $(P\circ T_\alpha)\subset B_{\mathcal P_{s,f,0}(V,W)}$ does the job:
clearly each $P \circ T_\alpha$ is a finite type polynomial whose $\Hcb$-norm is at most that of $P$, so just as in the previous paragraph pointwise convergence implies $\tau_\gamma$ convergence by Remark \ref{rmk:tau_0 topology}.
\end{proof}

At this point we can prove the desired transference result, compare to \cite[Prop. 5.7]{Mujica-linearization}.

\begin{theorem}\label{thm-CMAP-transfers}
    $V$ has the CMAP if and only if $\Gcb(B_V)$ has the CMAP.
\end{theorem}

\begin{proof}
    If $\Gcb(B_V)$ has the CMAP, by Proposition \ref{Prop:complemented subspace} it is clear that $V$ also has this property.

    Now, suppose that $V$ has the CMAP and consider the mapping $\delta_V\in \overline B_{\Hcb(B_V,\Gcb(B_V))}$. We know from Lemma \ref{lemma:approximation-by-polynomials} that there exists a net $(P_\alpha)\subset B_{\mathcal P_{s,f,0}(V,\mathcal{G}^\infty_{\cb}(B_V))}$ such that $\|P_\alpha (x)-\delta_V(x)\|\to 0$ for every $x\in B_V$. Linearizing these polynomials we obtain finite rank completely bounded maps $(\widehat{P_\alpha})$ with $\|\widehat{P_\alpha}\|_{\cb}\le 1$ such that the following diagram commutes
\[
\xymatrix{\mathcal{G}^\infty_{\cb}(B_V)\ar@{-->}[rd]^{\widehat{P_\alpha}}\\
     B_V\ar[r]^{P_\alpha}\ar^{\delta_V}[u]& \mathcal{G}^\infty_{\cb}(B_V) }
\]  

Note that $\widehat{P_\alpha}(\delta_V(x))\to \delta_V(x)$, for every $x\in B_V$. Then, we have that $\widehat{P_\alpha}(u)\to u$ for each $u\in \spa\{ \delta_V(x) \}_{x \in B_V}$.  Since the net $(\widehat{P_\alpha})$ is bounded, we also get pointwise convergence on the closure of $\spa\{ \delta_V(x) \}_{x \in B_V}$. Hence, $\mathcal{G}^\infty_{\cb}(B_V)$
 has the CMAP.
\end{proof}

To reach the more involved argument about the transference of OAP from $V$ to $\Gcb(B_V)$ we need some preparatory results about the action of the $\tau_\gamma$-topolgy on $\Hcb(B_V,W)$. The next step is an adaptation of \cite[Prop. 3.8]{aron2023linearization}, see also \cite[Prop. 4.9]{Mujica-linearization}.

\begin{proposition}\label{prop-topologies-equivalent-on-homogeneous-polynomials}
Let $V$ and $W$ be operator spaces.
Then $\tau_\gamma$ is finer than $\tau_0$ on $\Hcb(B_V,W)$, and for each $m\in\N$ these two topologies are equivalent on $\mathcal{P}_{s,0}(^mV,W)$.
\end{proposition}

\begin{proof}
Notice that it follows from Proposition \ref{prop-delta-is-cb-holomorphic} and Corollary \ref{corollary-Schwarz} that $(\delta_V)_\infty$ is continuous from $B_{\K(V)}$ to $B_{\K(\Gcb(B_V))}$.
Therefore, if $Y \subset B_{\K(V)}$ is compact, so is $(\delta_V)_\infty(Y) \subset B_{\K(\Gcb(B_V))}$.
By Propositions \ref{proposition-ball-is-convex-hull-of-image-of-delta} and  \ref{prop-Grothendieck-abs-mat-conv}, there exist sequences  $(a_j)$ and $(x_j)$ with $a_j>0$, $a_j \to 0$, $x_j \in B_{M_{n_j}(V)}$, such that $(\delta_V)_\infty(Y) \subseteq \overline{\amconv \{ a_j (\delta_V)_{n_j}(x_j) \} }^{\K(\Gcb(B_V))}$.
Therefore, for each $f\in\Hcb(B_V,W)$ we have,  using Lemma \ref{lemma-norm-of-abs-mat-conv-hull} once more, that
\[
\sup _{y \in Y} \n{ f_\infty(y) }_{\K(W)} \le \sup _{j} a_j \n{ f_{n_j}(x_j) }_{M_n(W)}
\]
and this proves that $\tau_\gamma$ is finer than $\tau_0$ on $\Hcb(B_V,W)$.

Now consider one of the seminorms that define $\tau_\gamma$:
\[
p(f) = \sup_{j} a_j \n{ f_{n_j}(x_j) }_{M_{n_j(W)}}
\]
where $n_j \in \N$, $x_j \in B_{M_{n_j}(V)}$, $a_j>0$ for all $j$ and $a_j \to 0$.\
By Remark \ref{remark-tau-gamma-with-sequences-bounded-by-1}, we can additionally assume that $a_j \le 1$ for all $j\in\N$.
For a homogeneous Schur polynomial $P\in \mathcal{P}_{s,0}(^mV,W)$ we have
\[
p(P) = \sup_j a_j \n{ P_{n_j}(x_j) } = \sup_j \n{  P_{n_j}(a_j^{1/m} x_j) } = \sup _{y \in Y} \n{ P_\infty(y) }
\]
where $Y = \{0\} \cup \big\{ a_j (\delta_V)_{n_j}(x_j) \big\}$ is a compact set in $B_{\K(V)}$ (note that this is where we need the additional assumption $a_j \le 1$, otherwise $a_j (\delta_V)_{n_j}(x_j)$ would not be guaranteed to be in $B_{\K(V)}$);
this shows that  $\tau_0$ is finer than $\tau_\gamma$ on  $\mathcal{P}_{s,0}(^mV,W)$.
\end{proof}

Next we adapt \cite[Prop. 3.9]{aron2023linearization}, see also \cite[Lemma 5.3]{Mujica-linearization}.

\begin{proposition} \label{prop-OAP-implies-polynomials-are-tau-gamma-dense}
	If $V$ has the OAP, for any $f\in \Hcb(B_V,W)$ there exists a net $(P_\alpha) $ of finite type Schur polynomials in $\mathcal{P}_{s,f,0}(V,W)$ such that $P_\alpha\overset{\tau_\gamma}{\to} f$.
\end{proposition}

\begin{proof}
Without loss of generality we can assume $f \in \overline{B}_{\Hcb(B_V,W)}$.
Moreover, by Lemma \ref{lemma:approximation-by-polynomials} it suffices to prove the result for each homogeneous polynomial $P\in\mathcal P_{s,0}(^mV,W)$ (for any $m$). By the OAP of $V$ along with Proposition \ref{prop-stable-point-norm} we know that there is a net $(T_\alpha)$  of finite rank mappings in $\CB(V,V)$ such that $(T_\alpha)_\infty\to Id_\infty=Id$ uniformly on the compact subsets of $\K(V)$. Now, by Remark \ref{rmk:polynomials}, $P_\alpha=P\circ T_\alpha$ are finite type $m$-homogeneous Schur polynomials. By the usual binomial formula, $(P_\alpha)_\infty\to P_\infty$ uniformly on the compact subsets of $B_{\K(V)}$, because $\|\widecheck{P}_\infty\| \le \frac{m^m}{m!}\|P_\infty\|$ and

$$
\|(P_\alpha)_\infty (u)-P_\infty(u)\|=\|P_\infty((T_\alpha)_\infty(u))-P_\infty(u)\|\le \sum_{k=0}^{m-1} \|\widecheck{P}_\infty\| \|u\|^k \|(T_\alpha)_\infty(u)-u\|^{m-k}.
$$

Now,  Proposition \ref{prop-topologies-equivalent-on-homogeneous-polynomials} implies that $P_\alpha\overset{\tau_\gamma}{\to} P$ finishing the proof.
\end{proof}

Finally we have all the ingredients to prove that the OAP transfers between $V$ and $\Gcb(B_V)$, an operator space version of \cite[Thm. 5.4]{Mujica-linearization}.

\begin{theorem}\label{thm-OAP-transfers}
	$V$ has the OAP if and only if $\Gcb(B_V)$ has the OAP.
\end{theorem}

\begin{proof}
If $\Gcb(B_V)$ has OAP, then so does $V$ because it is completely isometric to a completely contractively complemented subspace of $\Gcb(B_V)$ by  Proposition \ref{Prop:complemented subspace}.

Suppose now that $V$ has OAP.
	Consider $\delta_V \in  \Hcb(B_V,\Gcb(B_V))$ (see Proposition \ref{prop-delta-is-cb-holomorphic}). By Proposition \ref{prop-OAP-implies-polynomials-are-tau-gamma-dense}, there exists a net $(P_\alpha)\subset \mathcal P_{s,f,0}(V,\Gcb(B_V))$ such that $P_\alpha\overset{\tau_\gamma}{\to} \delta_V$. By the isomorphism given in Theorem \ref{thm-isomorphismtau-gamma-tau0}, we have that $(\widehat{P_\alpha})\subset \CB(\Gcb(B_V),\Gcb(B_V))$ is a net of finite rank linear mappings (see also Remark \ref{remark-linealization-preserves-properties}) satisfying $\widehat{P_\alpha}\overset{\tau_0}{\to} Id_V$, which means that $\Gcb(B_V)$ has OAP.
\end{proof}

\section*{Acknowledgements}

We would like to thank our friends Maite Fernández-Unzueta for her generous input on early versions of this work, and Luis Carlos Garc\'ia-Lirola for explaining to us the arguments referred to in the proof of \cite[Lemma 3.4]{aron2023linearization}.

\bibliography{references}
\bibliographystyle{abbrv}

\end{document}